\documentclass[a4paper]{article}

\usepackage[utf8]{inputenc}
\usepackage[T1]{fontenc}

\usepackage{amsmath}
\usepackage{amssymb}

\usepackage{xcolor}
\usepackage{colortbl}
\definecolor{hage}{rgb}{0.4,0.6,1}

\usepackage{hyperref}

\hypersetup{
    colorlinks = false,
    citebordercolor=black,
    linkbordercolor=hage,
}

\newcommand{\R}{\mathbb{R}}
\newcommand{\N}{\mathbb{N}}
\newcommand{\C}{\mathbb{C}}

\newcommand{\Ac}{\mathcal{A}}

\newcommand{\Pc}{\mathcal{P}}
\newcommand{\Lc}{\mathcal{L}}

\newcommand{\rk}{{\rm rk\,}}

\newcommand{\ran}{{\rm ran\,}}

\newcommand{\rank}{{\rm rank\,}}

\newcommand{\im}{{\rm Im\,}}

\usepackage[amsmath,amsthm,thmmarks]{ntheorem}

\theoremstyle{plain}
\newtheorem{defi}{Definition}[section]
\newtheorem{theorem}[defi]{Theorem}
\newtheorem{proposition}[defi]{Proposition}
\newtheorem{lemma}[defi]{Lemma}
\newtheorem{corollary}[defi]{Corollary}
\newtheorem{bsp}[defi]{Example}
\newtheorem{ass}[defi]{Assumption}

\title{Eigenvalue placement for regular matrix pencils with rank one perturbations
        }

\author{Hannes Gernandt\thanks{Institute of Mathematics, TU Ilmenau, Weimarer Stra\ss e 25, 98693 Ilmenau, Germany ({\tt hannes.gernandt@tu-ilmenau.de}).}
        \and Carsten Trunk\thanks{Institute of Mathematics, TU Ilmenau, Weimarer Stra\ss e 25, 98693 Ilmenau, Germany}. }

\begin{document}

\maketitle

\begin{abstract}
A regular matrix pencil $sE-A$ and its rank one perturbations are considered. We determine the sets in $\C\cup\{\infty\}$ which are the eigenvalues of the perturbed pencil. We show that the largest Jordan chains at each eigenvalue of $sE-A$ may disappear and the sum of the length of all destroyed Jordan chains is the number of eigenvalues (counted with multiplicities) which can be placed arbitrarily in $\C\cup\{\infty\}$.
We prove sharp upper and lower bounds of the change of the algebraic and geometric multiplicity of an eigenvalue under rank one perturbations.
Finally we apply our results to a pole placement problem  for a single-input differential algebraic equation with feedback.
\end{abstract}

\vspace{0.5cm}\\


%\pagestyle{myheadings}
%\thispagestyle{plain}
%\markboth{\textsc{H. Gernandt and C. Trunk}}{\textsc{Eigenvalue placement for %regular matrix pencils}}

\section{Introduction}
For square matrices $E$ and $A$ in $\C^{n\times n}$ we consider the matrix pencil
\begin{equation}\label{eins}
\mathcal A(s) := sE-A%\hage{\in\C[s]^{n\times n}}
\end{equation}
and study its set of eigenvalues $\sigma(\Ac)$, which is called the \textit{spectrum} of the matrix pencil. Here $\lambda\in\C$ is said to be an eigenvalue if $0$ is an eigenvalue of the matrix $\lambda E-A$ and we say that $\infty$ is an eigenvalue of $\Ac(s)$ if $E$ is not invertible. 

The spectral theory of matrix pencils is a generalization of the eigenvalue problem for matrices \cite{GLR09,M88,R89,SS90}. Recently, there is a growing interest in the spectral behavior under low rank perturbations of matrices \cite{BMRR15,DM03, MMRR09, RW11, S03,S04} and of matrix pencils \cite{B14, DM15, DMT08, EGR03, MMW15}. For matrices it was shown in \cite{DM03} that under generic rank one perturbations only the largest Jordan chain at each eigenvalue might be destroyed. Here a generic set of perturbations is a subset of $\C^{n \times n}$ which complement is a proper algebraic submanifold. 

We consider only \textit{regular} matrix pencils $\Ac(s)=sE-A$ which means that  the \textit{characteristic polynomial} $\det(sE-A)$ is not zero. Otherwise we call $\Ac(s)$ \textit{singular}. Jordan chains for regular matrix pencils at $\lambda$ correspond to Jordan chains of the matrix $J$ at $\lambda$, for  $\lambda\neq\infty$, or to Jordan chains of $N$ at $0$, for $\lambda=\infty$, in the Weierstra\ss\ canonical form \cite{G59},
\begin{align*}
%\label{wnf}
S(sE-A)T=s\begin{pmatrix}I_r & 0 \\ 0 & N\end{pmatrix}-\begin{pmatrix}J& 0 \\ 0 & I_{n-r}\end{pmatrix},\quad r\in\{0,1,\ldots,n\}%\quad J\in\K^{r\times r},~N\in\K^{(n-r)\times(n-r)}
\end{align*}
with $J\in\C^{r\times r}$ and $N\in\C^{(n-r)\times(n-r)}$ in Jordan canonical form, %over $\K$
 $N$ nilpotent and invertible $S,T\in\C^{n\times n}$. 
%Now the chains of $\Ac(s)$ at $\lambda\in\sigma(\Ac)\setminus\{\infty\}$ are given by the Jordan chains of the matrix $J$ at $\lambda$ and the chains of $\Ac(s)$ at $\infty$  are given by the Jordan chains of $N$ at $0$, see Proposition \ref{extralemma}. 
%We call the subspace of all chains corresponding to an eigenvalue $\lambda\in\sigma(\Ac)$ the root space $\mathcal{L}_{\lambda}(\Ac)$. 
In \cite{DMT08} it was shown that under generic rank one perturbations only the largest chain at $\lambda$ is destroyed. 

In this paper we investigate the behaviour of the spectrum of a regular matrix pencil
$\mathcal A(s)$ under a rank one perturbation $\mathcal P(s) := sF-G$.
As we will see in Section \ref{secDS}, the rank one condition allows us to write $\mathcal P(s)$ in the form
\begin{equation}\label{Peh}
\mathcal P(s) = (su + v)w^*\quad \mbox{or}\quad \mathcal P(s) = w(su^* + v^*)
\end{equation}
with non-zero vector $w$ and vectors $u$ and $v$ such that at least one of the two
is not zero.
%\marginpar{Mehr Referenzen zu low rank und Spektraltheorie für Pencil}
Rank one perturbations of the above form
are considered in design problems for electrical circuits
 where the entries of $E$ are determined by
the capacitances of the circuit. The aim is to improve the frequency behavior
by adding additional capacitances between certain nodes.
This corresponds, within the model,
to a (structured) rank one perturbation of the matrix $E$,
see \cite{BHK15,BTW16,HKSSTW11}.
%\marginpar{Hauptproblem: Es sollten alle verwendeten begriffe auch hier definiert werden}
Here we follow a more general approach and obtain the following results:
\begin{itemize}
\item[\rm (i)]
We find for unstructured rank one perturbations of the form \eqref{Peh}
 sharp lower and upper bounds for the dimension of the root subspace $\mathcal {L}_\lambda(\mathcal A+ \mathcal P)$ of the perturbed pencil at $\lambda$ in terms of the dimension of $\mathcal L_\lambda(\mathcal A)$, where $\mathcal{L}_{\lambda}(\Ac)$ denotes the subspace of all Jordan chains of the matrix pencil $\Ac(s)$ at the eigenvalue $\lambda$. More precisely, if $m_1(\lambda)$ denotes the length of the longest chain of $\mathcal A(s)$ at  $\lambda$ and let $M(\mathcal A)$ be the sum of all $m_1(\lambda)$ over all eigenvalues of $\Ac(s)$, that is,
 $M(\mathcal A) = \sum_{\lambda\in \sigma(\mathcal A)} m_1(\lambda)$. Then the bounds are 
 \begin{equation}\label{Rooot}
 \dim\mathcal L_\lambda (\mathcal A) -m_1(\lambda) \leq
 \dim \mathcal L_\lambda(\mathcal A+ \mathcal P) \leq
 \dim\mathcal L_\lambda(\mathcal A) +M(A)-m_1(\lambda).
\end{equation}
\item[\rm (ii)] We show a statement on the eigenvalue placement for regular matrix pencils which is our main result: $M(\mathcal A)$ eigenvalues (counted with multiplicities) can be placed arbitrarily in the complex plane. Either by creating new eigenvalues with
    one chain only or by adding one new (or extending an old) 
    chain at existing eigenvalues of $\mathcal A(s)$. Here the term new
    eigenvalue is understood in the sense that this value is an eigenvalue
    of $(\mathcal A + \mathcal P)(s)$ but not of $\mathcal A(s)$. %The number
    %$M (\mathcal A)$  of placeable eigenvalues
    %(counted  with multiplicities)
    %  equals the sum
    %of all lengths of the longest chains at all eigenvalues
    %of $\mathcal A(s)$.
 In addition we  obtain the same result for real matrices and real rank one perturbations.
\end{itemize}
%More precisely, if $m_1(\lambda)$ denotes the length of the longest chain
% of $\mathcal A(s)$ at  $\lambda$, then the number $M(\mathcal A)$ from (ii)
% corresponds to the sum of all $m_1(\lambda)$ over all eigenvalues of $\Ac$, that is,
% $M(\mathcal A) = \sum_{\lambda\in \sigma(\mathcal A)} m_1(\lambda)$
%and the statement in (i) can be written in the form
Roughly speaking, the behaviour of the spectrum under rank one
perturbations described in (i) and (ii) can be summarized in the
following way: At each eigenvalue of $\mathcal A(s)$  the
longest chain (or parts of it) may disappear but the remaining chains
at that eigenvalue
are then Jordan
chains of $(\mathcal A +\mathcal P)(s)$.
Moreover, the sum $M(\mathcal A)$ of the length of all the longest
chains is then the upper bound for the placement of new chains,
either at existing eigenvalues or at new eigenvalues. Those
new chains have to satisfy two rules. At new eigenvalues there
is only one chain with maximal length $M(\mathcal A)$ and at existing
eigenvalues at most one new chain may appear, again with maximal length $M(\mathcal A)$.
For a precise description of this placement result we refer to Theorem
\ref{pert} below.

The left hand side of \eqref{Rooot} is well-known \cite{B14,DMT08}. %\marginpar{die Abschätzung, oder was?}
In the case of matrices, i.e.\ $E=I_n$ in \eqref{eins} and $u=0$ in \eqref{Peh},
we refer to \cite{S03,S04} and to \cite{BLMPT15,HM94} for operators.
Results similar to (i) are known in the literature for generic low-rank perturbations  \cite{B14,DM03,DMT08,MMRR09,S03}. In the generic case it was shown in \cite{DMT08} that only the largest chain at each eigenvalue is eventually destroyed. In Proposition \ref{trunk} we show a non generic result that gives a  bound on the change of the Jordan chains of length $k$ for all  $k\in\N\setminus\{0\}$. Similar bounds were previously obtained for matrices in \cite{S03} and for operators in \cite{BLMPT15}.
%Here a generic set is a set with a complement which is
% a proper algebraic submanifold in the set of $n\times n$ matrix pairs.

 In special cases the placement problem in (ii) is considered in the literature.
 For $E$ positive definite
and $A$ symmetric the placement problem was studied in \cite{EGR03}. In the matrix case,
i.e. $E=I_n$ and $u=0$, the placement problem was solved for symmetric $A$ in \cite{G73}.
In \cite{K92} a related inverse problem was studied: For two given subsets of the complex plane,  two matrices were constructed whose set of eigenvalues equal these sets and  the matrices differ by rank one. All these eigenvalue placement settings above are special cases of our result in Section \ref{mainsection} below. 

In Section \ref{rest} we investigate the eigenvalue placement under parameter restrictions in the perturbation, i.e. in the representation \eqref{Peh} we fix $u, v\in\C^n$. This allows us to derive a sharper bound as in \eqref{Rooot}.
For these restricted placement problems, we obtain simple conditions on the number of eigenvalues that can be assigned arbitrarily.

In the final section  we present an application. We consider
the pole assignment problem under state feedback for single
input differential-algebraic equations. This problem is well studied in the literature \cite{C81,KZ88,L86,LO85,M93,R70} even for singular
matrix pencils, see \cite{D11} and the references therein. However, for single input systems we can view this problem as a parameter restricted rank one
perturbation problem from Section \ref{rest}.

\section{Eigenvalues and Jordan Chains of Matrix Pencils}
In this section the notion of eigenvalues and Jordan chains for matrix pencils $\Ac(s)=sE-A$ with $E,A\in\C^{n\times n}$ is recalled. Furthermore we summarize some basic spectral properties which are implied by the well known Weierstra\ss\ canonical form \cite{G59}.
%In the following we consider for matrices $E,A\in\C^{n\times n}$ matrix pencils of the form
%\begin{align}
%\label{pen}
%\mathcal{A}(s):=sE-A,\quad s\in\C,
%\end{align}
%and study the behavior of their eigenvalues under rank one perturbations.
 %Especially, $\rk\mathcal{A}=r$ for $r\in\N$ implies $\max\{\max_{s\in\C}\rk\mathcal{A}(s),\rk E\}= r$.\\
%We further assume \eqref{pen} to be \textit{regular}, which means it has full rank,  equal to $n$. %Equivalent to this is that the \textit{characteristic polynomial}
%\[
%\chi_{\Ac}(s):=\det(sE-A)
%\]
%is unequal to zero. %For $\Ac(s)=sI_n-A$ we simply write $\chi_{A}$ instead of $\chi_{\Ac}$.
%In the following it is useful to define $\mathcal{A}(\infty):=E$.
%To introduce the spectrum of matrix pencils we first recall the spectrum for matrices $A\in\C^{n\times n}$
%\[
%\sigma(A):=\{\lambda\in\C | \lambda I-A~ \text{is not invertible}\}
%\]
%The \textit{spectrum} of $\Ac$ is defined as
%\[
%\sigma(\mathcal{A}):=\{\lambda\in\C ~|~ \mathcal{A}(\lambda)~ \text{is not %invertible}\}.
%\]
%For pencils it is also useful to look at the extended spectrum

For fixed $\lambda\in\C$ observe that $\Ac(\lambda)$ is a matrix over $\C$. Hence the \textit{spectrum} of the matrix pencil $\Ac(s)=sE-A$ is defined as
\[
\sigma(\Ac):=\{\lambda\in\C ~|~ 0~\text{is an eigenvalue of $\Ac(\lambda)$}\},\quad \text{if  $E$ is invertible,}
\]
and
\[
\sigma(\Ac):=\{\lambda\in\C ~|~ 0~\text{is an eigenvalue of $\Ac(\lambda)$}\}\cup\{\infty\},\quad \text{if $E$ is singular.}
\]
Obviously the spectrum of a matrix pencil is a subset of the extended complex plane $\overline{\C}:=\C\cup\{\infty\}$ and
 the roots of the characteristic polynomial $\det(sE-A)$ are exactly the elements of $\sigma(\Ac)\setminus\{\infty\}$. %We call $\Ac(s)=sE-A$ regular if the characteristic polynomial $\det(sE-A)$ is not the zero polynomial.
Hence the spectrum of regular matrix pencils consists of finitely many points. For $\Ac(s)$ singular one always has $\sigma(\Ac)=\overline{\C}$.
%One can introduce the \textit{characteristic polynomial}
%\[
%\chi_{\Ac}(s):=\det\mathcal{A}(s)=\det(sE-A)\in\C[s].
%\]
%For $\Ac(s)=sI_n-A$ we simply write $\chi_{A}$ instead of $\chi_{\Ac}$.
%The assumption that $\Ac$ is regular is equivalent to $\chi_{\Ac}\neq 0$.
%\[
%\sigma(\mathcal{A})\setminus\{\infty\}=\{\lambda_1,\ldots,\lambda_k\}=\{\lambda\in\C ~|~ \chi_{\Ac}(\lambda)=0\}.
%\]
%So to investigate the spectrum of $\Ac$, one has to look at the roots of $\chi_{\Ac}$. %The \textit{algebraic multiplicity} of $\lambda\in\sigma(\Ac)$ is defined to be the multiplicity of the root $\lambda$ in $\chi_{\Ac}$.\\
%Next we investigate the structure of the eigenspaces.
%There is also a notion for generalized eigenvectors and eigenspaces in the context of operator polynomials (cf. \cite[Paragraph 11.2]{M88}). Since matrix pencils are linear operator polynomials we get the following:

We recall the notion for Jordan chains and root subspaces \cite[Section 1.4]{GLR09}, \cite[\S 11.2]{M88}).
The set $\{g_0,\ldots,g_{m-1}\}\subset \C^n$ is a %\textit{chain of generalized eigenvectors} or
\textit{Jordan chain} of \textit{length} $m$ at $\lambda\in\C$  if $g_0\neq0$ and
%there exists a vector $g_0\in\C^n\setminus\{0\}$ with $g_0\in\ker(\lambda E-A)$. We can also consider a chain of generalized eigenvectors $g_1,\ldots g_{m-1}\in\C^n\setminus\{0\}$ associated with $g_0$ which must per definition satisfy (cf. Markus88, §11.2)
%\[
%\sum_{k=0}^j\frac{\Ac^{(k)}(\lambda)}{k!}g_{j-k}=0,\quad \text{for $j=1,\ldots,m-1$}.
%\]
%Since $\Ac(s)=sE-A$, we have $\Ac'(s)=E$ and  the derivatives of order $k\geq2$ vanish. The above is therefore equivalent to
\begin{align*}
 (A-\lambda E)g_0=0,\quad (A-\lambda E)g_1=Eg_0,\quad \ldots,\quad (A-\lambda E)g_{m-1}=Eg_{m-2}
%\\&\lambda=\infty:& Eg_1&=0,& Ex_2&=Ax_1,\quad\ldots,& Ex_k&=Ax_{k-1}.
\end{align*}
%The set $\{g_0,\ldots,g_{m-1}\}$ is called a Jordan chain of length $m$ at $\lambda$.
and we call $\{g_0,\ldots,g_{m-1}\}\subset\C^n$ a Jordan chain of length $m$ at $\infty$ if
\begin{align*}
 g_0\neq0,\quad Eg_0=0,\quad Eg_1=Ag_0,\quad \ldots,\quad Eg_{m-1}=Ag_{m-2}.
%\\&\lambda=\infty:& Eg_1&=0,& Ex_2&=Ax_1,\quad\ldots,& Ex_k&=Ax_{k-1}.
\end{align*}
%it is a Jordan chain for $sA-E$ of length $m$ at $\lambda=0$.
%Observe that also the pencil $sA-E$ is regular, therefore this is well defined.\\ Wurde oben nicht vorausgesetzt!!!
Two Jordan chains $\{g_0,\ldots,g_k\}$ and $\{h_0,\ldots,h_l\}$ at $\lambda\in\overline{\C}$ are called \textit{linearly independent}, if the vectors $g_0,\ldots g_k, h_0,\ldots,h_l$ are linearly independent. Furthermore, we say that $\Ac(s)$ has $k$ Jordan chains of length $m$ if there exist $k$ linearly independent Jordan chains of length $m$ at $\lambda\in\overline{\C}$.
We denote for $\lambda\in\overline{\C}$ and $l\in\N\setminus\{0\}$ the subspace of all elements of all Jordan chains up to the length $l$ at $\lambda$ by
\[
\Lc_{\lambda}^l(\Ac):=\Big\{g_{j}\in\C^n ~|~ 0\leq j\leq l-1,~ \text{$\{g_0,\ldots,g_j\}$ is a Jordan chain at $\lambda$}\Big\}
\]
and the root subspace which consists of all elements of all Jordan chains at $\lambda$,
\[
\Lc_{\lambda}(\Ac):=\bigcup_{l=1}^{\infty}\Lc_{\lambda}^l(\Ac).
\]
It is well known that regular pencils $\Ac(s)=sE-A$ %with $A,E\in\K^{n\times n}$ for $\K=\R$ or $\K=\C$
can be transformed into the Weierstra\ss\ canonical form %over $\K$
\cite[Chapter XII, \S 2]{G59}, i.e.\ there exist invertible matrices $S,T\in\C^{n\times n}$ and  $r\in\{0,1,\ldots,n\}$ such that
\begin{align}
\label{wnf}
S(sE-A)T=s\begin{pmatrix}I_r & 0 \\ 0 & N\end{pmatrix}-\begin{pmatrix}J& 0 \\ 0 & I_{n-r}\end{pmatrix},%\quad J\in\K^{r\times r},~N\in\K^{(n-r)\times(n-r)}
\end{align}
with $J\in\C^{r\times r}$ and $N\in\C^{(n-r)\times(n-r)}$ in Jordan canonical form %over $\K$
and $N$ nilpotent.
From the Weierstra\ss\ canonical form, we have some well known properties \cite{BTW16,G59}.
\begin{proposition}
\label{extralemma}
For a regular matrix pencil $\Ac(s)=sE-A$ with Weierstra\ss\ canonical form \eqref{wnf} the following holds.
%\[
%\sigma(\Ac)\setminus\{\infty\}=\sigma(J)
%\]
%holds and
\begin{itemize}
\item[\rm (a)]
A Jordan chain $\{g_0,\ldots,g_{m-1}\}$ of the matrix pencil at $\lambda\in\C$ of length $m$ corresponds to a Jordan chain $\{\pi_rT^{-1}g_0,\ldots,\pi_rT^{-1}g_{m-1}\}\subset\C^r$ of $J$ at $\lambda$ of length $m$. Here $\pi_r$ denotes the projection of $x\in\C^n$ onto the first $r$ entries. Vice versa a Jordan chain $\{h_0,\ldots,h_{m-1}\}$ of $J$ at $\lambda$ corresponds to a Jordan chain $\left\{T\begin{pmatrix}h_0 \\0\end{pmatrix},\ldots,T\begin{pmatrix}h_{m-1} \\0\end{pmatrix}\right\}$ of the matrix pencil at $\lambda$.

\item[\rm (b)] A Jordan chain $\{g_0,\ldots,g_{m-1}\}$ of the matrix pencil at $\infty$ of length $m$ corresponds to a Jordan chain $\{\pi_{n-r}T^{-1}g_0,\ldots,\pi_{n-r}T^{-1}g_{m-1}\}\subset\C^{n-r}$ of $N$ at $0$ of length $m$. Here $\pi_{n-r}$ denotes the projection of $x\in\C^n$ onto the last $n-r$ entries. Vice versa a Jordan chain $\{h_0,\ldots,h_{m-1}\}$ of $N$ at $0$ corresponds to a Jordan chain $\left\{T\begin{pmatrix}0\\ h_0 \end{pmatrix},\ldots,T\begin{pmatrix}0\\ h_{m-1} \end{pmatrix}\right\}$ of the matrix pencil at $\infty$.%In the same way, the Jordan chains of $\Ac$ at $\infty$ correspond to the chains of $N$ at $0$.
\item[\rm (c)] A Jordan chain $\{g_0,\ldots,g_{m-1}\}$ of the matrix pencil at $\infty$ of length $m$ corresponds to a Jordan chain $\{h_0,\ldots,h_{m-1}\}$ of the dual pencil $\Ac'(s)=-sA+E$ at $0$.
\item[\rm (d)] If $r\geq 1$ we have $\sigma(\Ac)\setminus\{\infty\}=\sigma(J)$ and the characteristic polynomial of $sE-A$ is divisible by the minimal polynomial $m_J(s)$ of $J$ with
\begin{align}
\label{diefaktorisierung}
\det(sE-A)=(-1)^{n-r}\det(ST)^{-1}m_J(s)q(s),
\end{align}
where $q(s)$ is a monic polynomial of degree $r-\deg m_J$. The value $\lambda\in\sigma(\Ac)\setminus\{\infty\}$ is a root of $q(s)$ if and only if $\dim\ker\Ac(\lambda)\geq 2$.
%\hage{If $r=0$, then equation \eqref{diefaktorisierung} still holds when we define $m_J(s):=1$.} 
Moreover the multiplicity of a root $\lambda$ of $\det(sE-A)$ is equal to $\dim\mathcal{L}_{\lambda}(\Ac)$ and we have
\begin{align}
\label{summegleichn}
\sum_{\lambda\in\sigma(\Ac)}\dim\mathcal{L}_{\lambda}(\Ac)=n.
\end{align}
\end{itemize}
\end{proposition}
%Because of the identification between the Jordan chains of $\Ac$ and $J$ given in Proposition \ref{extralemma}, it is clear that the Jordan blocks of $J$ correspond to linearly independent Jordan chains of $\Ac$ which allows us to describe the Jordan structure of $\Ac$.
%To describe the Jordan structure we introduce
%%%Definition k(\lambda)%%%%%
%For each eigenvalue $\lambda\in\sigma(\Ac)$ we set
%\[
%k(\lambda):=\dim\ker\Ac(\lambda)
%\]
There are $\dim\ker\Ac(\lambda)$ linearly independent Jordan chains at $\lambda$ and, by Proposition \ref{extralemma}, this corresponds to the number of linearly independent Jordan chains of $J$ at $\lambda\neq\infty$ or $N$ at $0$ for $\lambda=\infty$. Each of these $\dim\ker\Ac(\lambda)$ different Jordan chains has a length which we denote by $m_j(\lambda)$, $1\leq j\leq \dim\ker\Ac(\lambda)$. These numbers $m_j(\lambda)$ are not uniquely determined, more precisely, they depend on the chosen Weierstra\ss\ canonical form \eqref{wnf} but they are unique up to permutations. In the following, we will choose those numbers in a specific way and we fix this in the following assumption.
%and denote by $m_j(\lambda)$, $1\leq j\leq k(\lambda)$, the size of the Jordan blocks of $J$ at $\lambda\in\sigma(\Ac)\setminus\{\infty\}$. Denote also by $m_j(\infty)$, $1\leq j \leq k(\infty)$ the size of the Jordan blocks of $N$ at $0$.
%Here the index $r\in\N$ and the Jordan canonical forms of $J$ and $N$ are independent of the transformation $S,T\in\K^{n\times n}$ up to permutations of the Jordan blocks.
%It is well known (cf.\ \cite{G59}) that $\sigma(J)=\sigma(\Ac)\setminus\{\infty\}$ and

%Introducing
%\[
%k(\lambda):=\dim\ker\Ac(\lambda),\quad \lambda\in\overline{C},
%\]
%the Jordan canonical form of $J$ and $N$ for $\K=\C$ can be written as
%\begin{align}
%\label{JN}
%J=\bigoplus_{\lambda\in\sigma(J)}\bigoplus_{j=1}^{k(\lambda)}J_{m_j(\lambda)}(\lambda),\quad N=\bigoplus_{j=1}^{k(\infty)}J_{m_j(\infty)}(0)
%\end{align}
%with the Jordan blocks $J_{m_j(\lambda)}(\lambda)$ of size $m_j(\lambda)$, $1\leq j\leq k(\lambda)$.

%\[
%J_{l}(\lambda)=\begin{pmatrix}\lambda & 1 &  & & \\ & \cdot & \cdot  &  &  \\ &&\cdot&\cdot& \\ & & &\cdot &1 \\ &  &   & & \lambda \end{pmatrix}\in\C^{l\times l}.
%\]
%Here $k(\lambda)$ is the number of Jordan chains at $\lambda\in\sigma(\Ac)$ and $m_j(\lambda)$ for $1\leq j\leq k(\lambda)$ is the length of the $j$-th Jordan chain at $\lambda\in\sigma(J)$. The numbers $m_j(\lambda)$ are uniquely determined up to permutations.\\
%In the sequel we always assume that the following holds.
\begin{ass}
\label{nijann}
Given a regular pencil $\Ac(s)=sE-A$ which has Weierstra\ss\ canonical form \eqref{wnf} with $r\in\{0,1,\ldots,n\}$ and matrices $J\in\C^{r \times r}$ and $N\in\C^{(n-r)\times (n-r)}$. % of the form \eqref{JN}.
%Then
%\[
%\sigma(\Ac)\setminus\{\infty\}=\sigma(J)=\{\lambda_1,\ldots,\lambda_m\}
%\]
%and we assume  for all $i=1,\ldots,m$ that the numbers $m_j(\lambda_i)$, $1\leq j\leq k(\lambda_i)$ are sorted in a non-decreasing order
Then we assume that for $\lambda\in\sigma(\Ac)$ the numbers $m_j(\lambda)$, $1\leq j\leq \dim\ker\Ac(\lambda)$, are sorted in a non-decreasing order
\begin{align}
\label{wirklichnijann}
m_{1}(\lambda)\geq\ldots\geq m_{\dim\ker\Ac(\lambda)}(\lambda).
\end{align}
\end{ass}
Observe that Assumption \ref{nijann} is no restriction for regular pencils. This means that for every regular pencil the matrices $S$ and $T$ in \eqref{wnf} can be chosen in such a way that the
Jordan blocks of $J$ %$J_{m_j(\lambda)}$ in \eqref{JN}
satisfy the condition \eqref{wirklichnijann}, see \cite{G59}.
Therefore with Assumption \ref{nijann} the minimal polynomial $m_J(s)$ of $J$ can be written as
\begin{align*}
m_J(s)=\prod_{\lambda\in\sigma(J)}(s-\lambda)^{m_1(\lambda)}%\begin{cases}\prod_{\lambda\in\sigma(J)}(s-\lambda)^{m_1(\lambda)},& \text{for $\K=\C$,}\\  \prod\limits_{\lambda\in\sigma(\Ac)\setminus\{\infty\}, \im\lambda\geq 0}((s-\re\lambda)^2+(\im\lambda)^2)^{m_1(\lambda)},& \text{for $\K=\R$.} \end{cases}
\end{align*}
and we introduce 
\begin{align}
\label{minpol}
m_{\Ac}(s):=\begin{cases}\prod_{\lambda\in\sigma(J)}(s-\lambda)^{m_1(\lambda)}, & \text{for $r\geq 1$,} \\
1, & \text{for $r=0$.}
\end{cases}
\end{align}
Note that with this definition the equation \eqref{diefaktorisierung} also holds for $r=0$ after replacing $m_J(s)$ by $m_{\Ac}(s)$.
\section{The structure of rank one pencils}
\label{secDS}
In this section we study pencils of rank one. Recall that the rank of a pencil $\Ac(s)$ is the largest $r\in\N$ such that $\Ac(s)$, viewed as a matrix with polynomial entries, has minors of size $r$  that are not the zero polynomial  \cite{DMT08,G59}.
This implies that $\Ac(s)$ has rank equal to $n$ if and only if $\Ac(s)$ is regular.
Hence, pencils of rank one are not regular for $n\geq 2$, meaning that they cannot be transformed to Weierstra\ss\ canonical form. Nevertheless, there is a simple representation given in the following proposition.
\begin{proposition}
\label{DS}
The pencil $\Pc(s)=sF-G$ with $F,G\in\C^{n\times n}$  has rank one if and only if there exists $u,v,w\in\C^n$ with $w\neq 0$ and ($u\neq0$ or $v\neq0$) such that
\begin{align}
\label{dsrank1}
\Pc(s)=(su+v)w^* \quad \text{or}\quad \Pc(s)=w(su^*+v^*).  %\Ac(s)=u_1(sv_1^*+v_2^*).
\end{align}
If $F,G$ are real matrices, then $u,v,w$ can be chosen to have real-valued entries. %there exist $u,v,w\in\R^n$ such that \eqref{dsrank1} holds.
\end{proposition}
\begin{proof}
Given that $\Pc(s)$ has rank one, then all minors of $\Pc(s)$ of size strictly larger than one  are the zero polynomial. Since 
\begin{align}
\label{rangkleinereins}
\rk(\Pc(\lambda))=\rk (\lambda F-G)\leq 1\quad \text{for all $\lambda\in\C$,}
\end{align}
 for $\lambda=0$ we have $\rk(G)\leq 1$. Then there exist  $u,v\in\C^n$ with $G=uv^*$. For $\lambda=1$ we have $\rk (F-G)\leq 1$, so there exists $w,z\in\C^n$ with $F-G=wz^*$. Using the representations above we see
\[
2F-G=2(F-G)+G=2wz^*+uv^*.
\]
From \eqref{rangkleinereins}, for $\lambda=2$ we have that $\rk(2F-G)\leq 1$. If $u$ and $w$ are linearly independent  then $z=\alpha v$ or $v=\alpha z$ for some $\alpha\in\C$. Let $z=\alpha v$ (the case $v=\alpha z$ can be proven similarly). Then
\[
sF-G=s(uv^*+wz^*)-uv^*=(s(u+\alpha w)-u)v^*,
\]
therefore $\Pc(s)$ admits a representation as in \eqref{dsrank1}. Now, assume that $u$ and $w$ are linearly dependent. Let $u=\beta w$ for some $\beta\in\C$  (the case $w=\beta u$ can be proven similarly), then
\[
sF-G=s(uv^*+wz^*)-uv^*=w(s(\beta v^*+z^*)-\beta v^*)
\]
holds, hence \eqref{dsrank1} is proven. The converse statement is obvious. For $F,G\in\R^{n\times n}$ the arguments above remain valid after replacing $\C$ by $\R$ and $u^*,v^*,z^*$ and $w^*$ by $u^T, v^T, z^T$ and $w^T$.
\end{proof}
%Given  $u,v,w\in\K^n$ with $w\neq 0$ and $u\neq0$ or $v\neq0$, then the pencils given in \eqref{dsrank1} are unequal to zero and have obviously the property that all minors of size greater than one vanish. Hence $\Pc$ has rank equal to one.

The following example illustrates that both representations in \eqref{dsrank1} are necessary.
\begin{bsp}
A short computation shows that the matrix pencils
\begin{align*}
\Pc_1(s)&:=\begin{pmatrix} s+1 &s+1\\ 1 & 1\end{pmatrix}=\left(s\begin{pmatrix} 1 \\ 0\end{pmatrix}+\begin{pmatrix}1 \\ 1\end{pmatrix}\right)(1,1),\\
\Pc_2(s)&:=\begin{pmatrix} s+1 & 1 \\ s+1 & 1\end{pmatrix}=\begin{pmatrix}1\\1\end{pmatrix}\left(s(1,0)+(1,1)\right)
\end{align*}
admit only one of the representations given in Proposition \ref{DS}.
\end{bsp}

If in \eqref{dsrank1} the elements $u,v\in\C^n$ are linearly dependent both representations in \eqref{dsrank1} coincide and without restriction we can write for non-zero $(\alpha,\beta)\in\C^2$
\begin{align}
\label{dasP}
\Pc(s)=(\alpha s -\beta)uw^*.%\quad \text{or}\quad \Pc(s)=(\alpha s -\beta)wv^*.
\end{align}
%and both representations in \eqref{dsrank1} coincide.
The next lemma provides a simple criterion for $(\Ac+\Pc)(s)$ to be regular when $\Pc(s)$ is of the form \eqref{dasP}.
\begin{lemma}
\label{zusatzlemma}
Let $\Ac(s)=sE-A$ be regular. Choose $(\alpha,\beta)\in\C^2$ non-zero and let $\Pc(s)$ be given by \eqref{dasP}. Then the following holds.
\begin{itemize}
    \item[\rm (a)] Assume $\alpha\neq 0$. If $\beta/\alpha\in\sigma(\Ac)$ then  $\beta/\alpha\in\sigma(\Ac+\Pc)$. If  $\beta/\alpha\notin\sigma(\Ac)$ then $(\Ac+\Pc)(s)$ is regular.
    \item[\rm (b)] Assume $\alpha=0$. If $\infty\in\sigma(\Ac)$ then $\infty\in\sigma(\Ac+\Pc)$. If  $\infty\notin\sigma(\Ac)$ then $(\Ac+\Pc)(s)$ is regular.
\end{itemize}
\end{lemma}
\begin{proof}
For $\alpha\neq0$ we have
\begin{equation*}%\label{Lemma3.3}
\det(\Ac+\Pc)(\frac{\beta}{\alpha})=\det\left(\frac{\beta}{\alpha}E-A+(\alpha \frac{\beta}{\alpha} -\beta)uw^*\right)=\det\left(\frac{\beta}{\alpha}E-A\right)=\det\Ac(\frac{\beta}{\alpha})
\end{equation*}
and (a) follows. For $\alpha=0$ we use that $\infty\in\sigma(\Ac)$ if and only if the leading coefficient $E$ is singular. Since  $\alpha=0$, the leading coefficient of $(\Ac+\Pc)(s)$ is $E$ for all $u,w\in\C^n$, hence (b) is proved. 
\end{proof}
%Therefore the characteristic polynomial of $\Ac+\Pc$ is not the zero polynomial. Hence $\Ac+\Pc$ is regular.
%The following example illustrates that not every rank one pencil is of the form \eqref{dasP}.

%It is well known that there exist invertible matrices $S,T\in\C^{n\times n}$ transforming %$\Ac$ into Kronecker canonical form  (cf. \cite[Ch. XII]{G59})
%\begin{align}
%\label{knf}
%S\Ac(s)T=\diag(sI - J,I- s N)\oplus %\diag(L_{\varepsilon_1}(s),\ldots,L_{\varepsilon_d}(s),L_{\eta_1}(s)^*,\ldots,L_{\eta_d}(s)^*)
%\end{align}
%with
%\[
%L_{k}(s):=\begin{pmatrix} s & -1 && &\\ &s & -1 & &  \\ & & \ddots & \ddots &   \\ & &   & s &-1  \end{pmatrix}\in\C^{k\times (k+1)},\quad k\in\N\setminus\{0\}.
%\]
%Now, define
%\[
%\varepsilon:=\varepsilon_1+\ldots+\varepsilon_d\quad \text{and}\quad \eta:=\eta_1+\ldots+\eta_d,
%\]
%assume $\diag(sI - J,I- s N)\in\C^{\rho\times \rho}$ and let $r\in\N$ be the rank of $\Ac(s)$, then we conclude (cf. \cite{DMT08sing})
%\[
%\rho+\varepsilon+\eta=r\quad \text{and}\quad r=n-d.
%\]
%Since $r=1$ we conclude $\rho,\varepsilon_1,\eta_1\in\{0,1\}$ and without restriction $ \varepsilon_j=\eta_j=0$ for all $j=2,\ldots,d$. Rewriting \eqref{knf} by multiplying with $S^{-1}$ and $T^{-1}$ and using that by definition
%\[
%\diag(L_1(s),L_0(s)^*)=\begin{pmatrix}s & -1 \\ 0 & 0\end{pmatrix}\quad \text{and}\quad \diag(L_0(s),L_1(s)^*)=\begin{pmatrix}s & 0 \\ -1 & 0\end{pmatrix}
%\]
%holds, the claim easily follows.

\section{Change of the root subspaces under rank one perturbations} %root subspace under rank one perturbations}
In this section, we obtain bounds on the number of eigenvalues which can be changed by a rank one perturbation. %can be obtained for a given regular pencil \eqref{pen} after applying a rank one perturbation.\\
%First we describe the change  of the Jordan structure.
The following lemma is a special case ($r=1$) of \cite[Lemma 2.1]{DMT08}. % with Proposition \ref{extralemma}.
\begin{lemma}
\label{DM}
Let $\Ac(s)$ be a regular matrix pencil satisfying Assumption \ref{nijann} and let $\Pc(s)$ be of rank one. Assume that $(\Ac+\Pc)(s)$ is regular and let $\lambda\in\overline{\C}$ be an eigenvalue of $\Ac(s)$. Then $(\Ac+\Pc)(s)$ has at least $\dim\ker\Ac(\lambda)-1$ linearly independent Jordan chains at $\lambda$. For $\dim\ker\Ac(\lambda)\geq 2$ 
these chains can be sorted in such a way that for the length of the {\rm i}th chain $\tilde{m}_i(\lambda)$ the following holds
%These chains can be sorted by their length
%For the following we sort these $k(\lambda)-1$ chains by their length in a non-decreasing way and denote the corresponding lengths : The length of the $(i-1)$th chain of $\Ac+\Pc$ at $\lambda$ is denoted by $\tilde{m}_i(\lambda)$ for
% of length $\tilde{m}_i(\lambda)$ such that
\[
\tilde{m}_2(\lambda)\geq\ldots\geq \tilde{m}_{\dim\ker\Ac(\lambda)}(\lambda)\quad \text{and}\quad \tilde{m}_i(\lambda)\geq m_i(\lambda),\quad 2\leq i\leq \dim\ker\Ac(\lambda).
\]
\end{lemma}
The following result describes the maximal change of the root subspace dimension under rank one perturbations. For matrices, that is, if $E=I_n$, this result was obtained in \cite{S03}, see also \cite{BLMPT15,T80}. % but the idea goes back to \cite{T80}.
\begin{proposition}
\label{trunk}
Let $\Ac(s)$ be a regular matrix pencil satisfying Assumption \ref{nijann}%and %$\sigma(\Ac)\setminus\{\infty\}=\{\lambda_1,\ldots,\lambda_m\}$
%and transformable to Weierstra\ss\ normal form \eqref{wnf} with $J\in\K^{r\times r}$.
, then for any rank one pencil $\Pc(s)$ such that $(\Ac+\Pc)(s)$ is regular we have for all $\lambda\in\overline{\C}$ and $k\in\N\setminus\{0\}$
	\begin{align}
	\label{einprostufe}
	\left|\dim\frac{\mathcal{L}_{\lambda}^{k+1}(\Ac+\Pc)}{\mathcal{L}_{\lambda}^k(\Ac+\Pc)}-\dim\frac{\mathcal{L}_{\lambda}^{k+1}(\Ac)}{\mathcal{L}_{\lambda}^{k}(\Ac)}\right|\leq 1,\\[1ex]
		\label{trunkungl}
	|\dim \mathcal{L}_{\lambda}^k(\Ac+\Pc)-\dim \mathcal{L}_{\lambda}^k(\Ac)|\leq k.
	\end{align}
\end{proposition}
\begin{proof}
We prove the inequality \eqref{einprostufe}. Assume $\lambda\neq\infty$ and that for $k,l\in\N\setminus\{0\}$ we have
\begin{align}
\label{q}
\dim\frac{\mathcal{L}_{\lambda}^{k+1}(\Ac)}{\mathcal{L}_{\lambda}^k(\Ac)}=l\geq 2.
\end{align}
%By Proposition \ref{extralemma} (a)
This is equivalent to the fact $\Ac(s)$ has $l$ linearly independent Jordan chains at $\lambda$ with length at least $k+1$ which means
\[
m_1(\lambda)\geq m_2(\lambda)\ldots\geq m_l(\lambda)\geq k+1.
\]
It follows from Lemma \ref{DM} that $(\Ac+\Pc)(s)$ has at least $l-1$ linearly independent Jordan chains with lengths
\[
\tilde{m}_2(\lambda)\geq\ldots\geq \tilde{m}_{l}(\lambda)\geq m_l(\lambda)\geq k+1
\]
which leads to
\begin{align}
\label{noch}
\dim\frac{\mathcal{L}_{\lambda}^{k+1}(\Ac+\Pc)}{\mathcal{L}_{\lambda}^k(\Ac+\Pc)}\geq l-1.
\end{align}
%Using \eqref{q} yields
%\[
%\dim\frac{\mathcal{L}_{\lambda}^{k+1}(\Ac)}{\mathcal{L}_{\lambda}^k(\Ac)}-\dim\frac{\mathcal{L}_{\lambda}^{k+1}(\Ac+\Pc)}{\mathcal{L}_{\lambda}^k(\Ac+\Pc)}\leq 1.
%\]
It remains to show that the expression in \eqref{noch} is less or equal to $l+1$.
Indeed, assume
\[
\dim\frac{\mathcal{L}_{\lambda}^{k+1}(\Ac+\Pc)}{\mathcal{L}_{\lambda}^k(\Ac+\Pc)}\geq l+2.
\]
If we consider the regular pencil $(\Ac+\Pc)(s)$ and the rank one pencil $-\Pc(s)$ and apply the above arguments, we have that
\[
\dim\frac{\mathcal{L}_{\lambda}^{k+1}(\Ac)}{\mathcal{L}_{\lambda}^k(\Ac)}\geq l+1,
\]
which is a contradiction to \eqref{q}. % and \eqref{einprostufe} is proved.
%, starting with the regular pencil $\Ac+\Pc$ and the rank one perturbation $-\Pc$ leads to the contradiction
%\[
%l=\dim\frac{\mathcal{L}_{\lambda}^{k+1}(\Ac)}{\mathcal{L}_{\lambda}^k(\Ac)}\geq l+1.
%\dim\frac{\mathcal{L}_{\lambda}^{k+1}(\Ac+\Pc)}{\mathcal{L}_{\lambda}^k(\Ac+\Pc)}-\dim\frac{\mathcal{L}_{\lambda}^{k+1}(\Ac)}{\mathcal{L}_{\lambda}^k(\Ac)}\leq 1.
%\]
%This proves \eqref{einprostufe}.
Hence, \eqref{einprostufe} is shown for $l\geq 2$. If $l=1$, i.e.\
\[
\dim\frac{\mathcal{L}_{\lambda}^{k+1}(\Ac)}{\mathcal{L}_{\lambda}^{k}(\Ac)}=1
\]
and assume that $\dim\frac{\mathcal{L}_{\lambda}^{k+1}(\Ac+\Pc)}{\mathcal{L}_{\lambda}^{k}(\Ac+\Pc)}\geq 3$. Lemma \ref{DM} applied to the regular matrix pencil $(\Ac+\Pc)(s)$ shows $\dim\frac{\mathcal{L}_{\lambda}^{k+1}(\Ac)}{\mathcal{L}_{\lambda}^{k}(\Ac)}\geq 2$, a contradiction and \eqref{einprostufe} follows.

%\hage{Assume now that we have $l=1$ in \eqref{q}. Then either $\dim\frac{\mathcal{L}_{\lambda}^{k+1}(\Ac+\Pc)}{\mathcal{L}_{\lambda}^{k}(\Ac+\Pc)}=1$ such that \eqref{einprostufe} holds trivially or $\dim\frac{\mathcal{L}_{\lambda}^{k+1}(\Ac+\Pc)}{\mathcal{L}_{\lambda}^{k}(\Ac+\Pc)}\geq 2$ such that the above arguments can be repeated for the regular matrix pencil $\Ac+\Pc$ instead of $\Ac$. This proves \eqref{einprostufe}}
Now we show \eqref{trunkungl}. For $k=1$ the definition of $\mathcal{L}^i_{\lambda}(\Ac)$ implies $\mathcal{L}^1_{\lambda}(\Ac)=\ker\Ac(\lambda)$. Since $\Ac(\lambda)$ and $(\Ac+\Pc)(\lambda)$ are matrices and $\Pc(\lambda)$ is a matrix of rank at most one (see Proposition \ref{DS}), the estimates
\begin{align*}
\rank(\Ac(\lambda))=\rank((\Ac+\Pc)(\lambda)-\Pc(\lambda))&\leq\rank((\Ac+\Pc)(\lambda))+\rank(\Pc(\lambda)),\\ \rank((\Ac+\Pc)(\lambda))&\leq\rank(\Ac(\lambda))+1
\end{align*}
imply $|\rank((\Ac+\Pc)(\lambda))-\rank(\Ac(\lambda))|\leq 1$ and together with the dimension formula $\dim\ker\Ac(\lambda)+\rk(\Ac(\lambda))=n$ this leads to
\begin{align}
\begin{split}
\label{kisteins}
&~~~~|\dim\ker\Ac(\lambda)-\dim\ker(\Ac+\Pc)(\lambda)|\\&=\big|n-\rank(\Ac(\lambda))-\big(n-\rk((\Ac+\Pc)(\lambda))\big)\big|\leq 1.
\end{split}
\end{align}
Therefore \eqref{trunkungl} holds for $k=1$.
For $k\geq2$ we have the identity
\[
\dim\mathcal{L}_{\lambda}^k(\Ac)=\dim\ker\Ac(\lambda)+\sum_{m=1}^{k-1}\dim\frac{\mathcal{L}^{m+1}_{\lambda}(\Ac)}{\mathcal{L}^{m}_{\lambda}(\Ac)}
\]
which leads to
\begin{align*}
&~~~~~|\dim\mathcal{L}_{\lambda}^k(\Ac)-\dim\mathcal{L}_{\lambda}^k(\Ac+\Pc)|\\ &\leq|\dim\ker\Ac(\lambda)-\dim\ker(\Ac+\Pc)(\lambda)|+\sum_{m=1}^{k-1}\left|\dim\frac{\mathcal{L}^{m+1}_{\lambda}(\Ac)}{\mathcal{L}^{m}_{\lambda}(\Ac)}-\dim\frac{\mathcal{L}^{m+1}_{\lambda}(\Ac+\Pc)}{\mathcal{L}^{m}_{\lambda}(\Ac+\Pc)}\right|
\end{align*}
and \eqref{kisteins} together with \eqref{einprostufe} imply \eqref{trunkungl}.

For $\lambda=\infty$ we consider the dual pencil $\Ac'(s)=-As+E$ at $\lambda=0$. Obviously $\Ac'(s)$ and the dual pencil $(\Ac+\Pc)'(s)$ of $(\Ac+\Pc)(s)$ are regular. By Proposition \ref{extralemma} (c), it remains to apply \eqref{einprostufe} and \eqref{trunkungl} for $\lambda=0$ to $\Ac'(s)$ and $(\Ac+\Pc)'(s)$ to see that \eqref{einprostufe} and \eqref{trunkungl} hold for $\Ac(s)$ and $(\Ac+\Pc)(s)$ at $\lambda=\infty$.
\end{proof}
%The dual pencil $\Ac'$ is obviously regular from the assumptions $\Ac'+\Pc'$ is also regular. Hence the previous arguments can be applied which proves the inequalities \eqref{einprostufe} and \eqref{trunkungl} in this case.

For $k=1$ the inequality %the definition of $\mathcal{L}^i_{\lambda}(\Ac)$ implies $\mathcal{L}^1_{\lambda}(\Ac)=\ker\Ac(\lambda)$ and
\eqref{trunkungl} leads to the following statement.
\begin{corollary}
\label{starr}
Let $\Ac(s)$ be a regular matrix pencil, then for any rank one pencil $\Pc(s)$ such that $(\Ac+\Pc)(s)$ is regular we have   	
\[
	\left\{\lambda\in\sigma(\Ac) ~\Big|~  \dim\ker\Ac(\lambda)\geq 2\right\}\subseteq\sigma(\Ac+\mathcal{P})
	\]
	and for every $\mu\in\sigma(\Ac+\Pc)\setminus\sigma(\Ac)$
	\begin{align*}
	%\label{einekette}
	\dim\ker(\Ac+\Pc)(\mu)=1,
	\end{align*}
	i.e., in this case, there is only one Jordan chain of length $\dim\mathcal{L}_{\mu}(\Ac+\Pc)$.
\end{corollary}

Proposition \ref{trunk} states, roughly speaking, that the largest possible change in the dimensions of $\mathcal{L}_{\lambda}(\Ac+\Pc)$ compared with $\mathcal{L}_{\lambda}(\Ac)$ is bounded by the length of the largest Jordan chain of $\Ac(s)$ and $(\Ac+\Pc)(s)$. %the maximal length of the Jordan chains in $\mathcal{L}_{\lambda}(\Ac+\Pc)$ and $\mathcal{L}_{\lambda}(\Ac)$.
However, the aim of the following Theorem \ref{pert} is to give bounds for the change of dimension of $\mathcal{L}_{\lambda}(\Ac+\Pc)$ only in terms of the unperturbed pencil $\Ac(s)$. For this, we use the number $m_1(\lambda)$ which is according to Assumption \ref{nijann} the length of the largest Jordan chain of $\Ac(s)$ at $\lambda$ %the length of the longest Jordan chain of $\Ac$ at $\lambda$
and the number
%By Assumption \ref{nijann} this length of the longest Jordan chain at $\lambda$ is given by $m_1(\lambda)$. Hence, to obtain upper and lower bounds for the change of the dimension of $\mathcal{L}_{\lambda}(\Ac+\Pc)$ we utilize the numbers $m_1(\lambda)$ and
%This implies the following bounds in terms of
\begin{align*}
%\label{Mdefi}
M(\Ac):=\sum_{\mu\in\sigma(\Ac)}m_1(\mu).
\end{align*}

\begin{theorem}
\label{pert}
Let $\Ac(s)$ be a regular matrix pencil satisfying Assumption \ref{nijann}. %and %$\sigma(\Ac)\setminus\{\infty\}=\{\lambda_1,\ldots,\lambda_m\}$
%and transformable to Weierstra\ss\ normal form \eqref{wnf} with $J\in\K^{r\times r}$.
Then for any rank one pencil $\Pc(s)$ such that $(\Ac+\Pc)(s)$ is regular we have for $\lambda\in\sigma(\Ac)$
%the following holds true. Annahme $\Ac+\Pc$ regular nötig??\\
%For an arbitrary rank one pencil $\Pc$ we have
	%and for all $i=1,\ldots,m$
	\begin{align}
	\label{maximal1}
	%m_1(\lambda)\leq|\dim\mathcal{L}_{\lambda}(\Ac+\Pc)-\dim\mathcal{L}_{\lambda}(\Ac)|\leq M(\Ac)+m_1(\lambda)
	\dim \mathcal{L}_{\lambda}(\Ac)-m_{1}(\lambda) \leq\dim \mathcal{L}_{\lambda}(\Ac+\Pc) \leq \dim \mathcal{L}_{\lambda}(\Ac)+M(\Ac)-m_{1}(\lambda),	%\\
	%\label{maximal2}
	%&\dim \mathcal{L}_{\infty}(\Ac)-n_{\infty,1} &\leq\dim \mathcal{L}_{\infty}(\Ac+\Pc) &\leq \dim \mathcal{L}_{\infty}(\Ac)+\deg m_J
	\end{align}
	whereas the change in the dimension for $\lambda\in\overline{\C}\setminus\sigma(\Ac)$ is bounded by
	\begin{align}
\label{14a}
	0\leq\dim\mathcal{L}_{\lambda}(\Ac+\Pc)\leq M(\Ac).
	\end{align}
	Summing up, we obtain the following bounds
	\begin{align}
	\begin{split}
	\label{maximal3}
	\sum_{\lambda\in\sigma(\Ac)}\dim\mathcal{L}_{\lambda}(\Ac+\Pc)\geq n-M(\Ac),\\ \sum_{\lambda\in\sigma(\Ac+\Pc)\setminus\sigma(\Ac)}\dim\mathcal{L}_{\lambda}(\Ac+\Pc) \leq M(\Ac).
	\end{split}
	\end{align}
\end{theorem}
\begin{proof}
 %From equation \eqref{trunkungl} in Proposition \ref{trunk} with $k=m_1(\lambda)$ and fixed $\lambda\in\overline{\C}$ we see that
 %\[
%\dim \mathcal{L}_{\lambda}^{m_1(\lambda)}(\Ac)-\dim \mathcal{L}_{\lambda}^{m_1(\lambda)}(\Ac+\Pc)\leq|\dim \mathcal{L}_{\lambda}^{m_1(\lambda)}(\Ac+\Pc)-\dim \mathcal{L}_{\lambda}^{m_1(\lambda)}(\Ac)|\leq m_1(\lambda)
 %\]
By Assumption \ref{nijann}, we have $\mathcal{L}_{\lambda}(\Ac)=\mathcal{L}_{\lambda}^{m_1(\lambda)}(\Ac)$. Then \eqref{trunkungl} implies for $\lambda\in\sigma(\Ac)$
 \begin{align}
 \begin{split}
 \label{thebound}
\dim \mathcal{L}_{\lambda}(\Ac)-m_1(\lambda) &=\dim \mathcal{L}_{\lambda}^{m_1(\lambda)}(\Ac)-m_1(\lambda) \\ &\leq  \dim \mathcal{L}_{\lambda}^{m_1(\lambda)}(\Ac+\Pc)\leq \dim \mathcal{L}_{\lambda}(\Ac+\Pc).
\end{split}
 \end{align}
 This is the lower bound in \eqref{maximal1}. Since $(\Ac+\Pc)(s)$ is regular we can apply \eqref{summegleichn}, \eqref{thebound} and the upper bound for $\lambda\in\sigma(\Ac)$ follows from
 \begin{align*}
 \dim\mathcal{L}_{\lambda}(\Ac+\Pc) &=n-\sum_{\mu\in\sigma(\Ac+\Pc)\setminus\{\lambda\}} \dim\mathcal{L}_{\mu}(\Ac+\Pc)\\&\leq\sum_{\mu\in\sigma(\Ac)} \dim\mathcal{L}_{\mu}(\Ac)-\sum_{\mu\in\sigma(\Ac)\setminus\{\lambda\}} \dim\mathcal{L}_{\mu}(\Ac+\Pc)\\
 &= \dim\mathcal{L}_{\lambda}(\Ac)+\sum_{\mu\in\sigma(\Ac)\setminus\{\lambda\}}\dim\mathcal{L}_{\mu}(\Ac)-\sum_{\mu\in\sigma(\Ac)\setminus\{\lambda\}}\dim\mathcal{L}_{\mu}(\Ac+\Pc)\\
 &\leq  \dim\mathcal{L}_{\lambda}(\Ac)+\sum_{\mu\in\sigma(\Ac)\setminus\{\lambda\}}m_1(\mu)=\dim\mathcal{L}_{\lambda}(\Ac)+M(\Ac)-m_1(\lambda).
 \end{align*}
 Hence \eqref{maximal1} is proved and applying the same estimates for $\lambda\in\overline{\C}\setminus\sigma(\Ac)$  proves \eqref{14a}. We continue with the proof of \eqref{maximal3}.
 Relation \eqref{thebound} implies
 \[
 \sum_{\lambda\in\sigma(\Ac)}\dim\mathcal{L}_{\lambda}(\Ac+\Pc)\geq\sum_{\lambda\in\sigma(\Ac)}(\dim\mathcal{L}_{\lambda}(\Ac)-m_1(\lambda))=n-\sum_{\lambda\in\sigma(\Ac)}m_1(\lambda)=n-M(\Ac)
 \]
 and this yields
 \[
 \sum_{\lambda\in\sigma(\Ac+\Pc)\setminus\sigma(\Ac)}\dim\mathcal{L}_{\lambda}(\Ac+\Pc)=n-\sum_{\lambda\in\sigma(\Ac)}\dim\mathcal{L}_{\lambda}(\Ac+\Pc)%\leq n-(n-M(\Ac))
 \leq M(\Ac).
 \]
\end{proof}
%The remaining statements of Theorem \ref{pert} follow directly from Proposition \eqref{trunk}.
%Assume $\dim\ker \Ac(\lambda)\geq 2$ and let $\Pc$ be a rank one pencil, then from
%Proposition \ref{trunk} we conclude $\dim\ker(\Ac+\Pc)(\lambda)\geq 1$, therefore $\lambda\in\sigma(\Ac+\Pc)$. In the same way we see that
%\[
%\dim\ker(\Ac+\Pc)(\mu)\leq 1,\quad \mu\in\overline{\C}\setminus\sigma(\Ac)
%\]
%which implies \eqref{einekette}.

From the inequality \eqref{maximal3} we see that the number of changeable eigenvalues under a rank one perturbation is bounded by $M(\Ac)$.

\section{Eigenvalue placement with rank one perturbations}
\label{mainsection}

In this section we study which sets of eigenvalues can be obtained by rank one perturbations.
The following theorem is the main result. It states that for a given set of complex numbers there exists a rank one perturbation $\Pc(s)$ such that the set is included in  $\sigma(\Ac+\Pc)$, provided the given set has not more than $M(\Ac)$ elements.
\begin{theorem}
\label{hauptsatz}
Let $\Ac(s)$ be a regular matrix pencil satisfying Assumption \ref{nijann} and choose  %$\sigma(\Ac)\setminus\{\infty\}=\{\lambda_1,\ldots,\lambda_m\}$ %and% $\sigma(\Ac)\setminus\{\infty\}=\{\lambda_1,\ldots,\lambda_m\}$ the following holds true.
 pairwise distinct numbers $\mu_1,\ldots,\mu_l\in\overline{\C}$ with $l\leq M(\Ac)$. Choose multiplicities $m_1,\ldots,m_l\in\N\setminus\{0\}$ with  $\sum_{i=1}^lm_i=M(\Ac)$. Then
%with $\frac{\beta}{0}:=\infty$.
the following statements hold true.
\begin{itemize}
\item[\rm (a)] There exists a rank one pencil $\mathcal{P}(s)=(\alpha s-\beta)uv^*$ with $\alpha\in\C\setminus\{0\}$, $\beta\in\C$ and $u,v\in\C^n$ such that $(\Ac+\Pc)(s)$ is regular,  and
\begin{align}
\label{mualternat}
\sigma(\Ac+\Pc)=\{\mu_1,\ldots,\mu_l\}\cup\{\lambda\in\sigma(\Ac) ~|~ \dim\ker\Ac(\lambda)\geq 2\}
\end{align}
%and denote again by $m(\mu_i)$ how often the value $\mu_i$ occurs in the listing $\mu_1,\ldots,\mu_l$ then for all rank one pencils $\Pc$ that satisfy \eqref{mualternat} we have
with multiplicities
\begin{align}
\label{main2}
\dim \mathcal{L}_{\lambda}(\Ac+\Pc)=\begin{cases}\dim \mathcal{L}_{\lambda}(\Ac)-m_1(\lambda)+m_i, & \text{for $\lambda=\mu_i\in\sigma(\Ac)$,} \\
\dim \mathcal{L}_{\lambda}(\Ac)-m_1(\lambda), & \text{for $\lambda\in\sigma(\Ac)\setminus\{\mu_1,\ldots,\mu_l\}$,} \\
m_i, & \text{for $\lambda=\mu_i\notin\sigma(\Ac)$,} \\
0, & \text{for $\lambda\notin\sigma(\Ac)\cup\{\mu_1,\ldots,\mu_l\}$.}
\end{cases}
\end{align}
%Here $(\alpha,\beta)\in\C^{2}\setminus\{(0,0)\}$ can be chosen arbitrarily, only requiring that
%\[
%\frac{\beta}{\alpha}\notin\{\mu_1,\ldots,\mu_l\}\cup\sigma(\Ac).
%\]
%with $\frac{\beta}{0}:=\infty$.
\item[\rm (b)] If $E$, $A$ are real matrices and  $\{\mu_1,\ldots,\mu_l\}$ is symmetric with respect to the real line with $m_i=m_j$ if $\mu_j=\overline{\mu_i}$ and all $i,j=1,\ldots,l$,  %and also that $\alpha,\beta\in\R$ holds. %$\beta/\alpha\in\R\cup\{\infty\}$ %$m(\mu_i)=m(\overline{\mu_i})$ for all  $i=1,\ldots,l$,
there exists $\alpha,\beta\in\R$ and $u,v\in\R^n$ such that $\mathcal{P}(s)=(\alpha s-\beta)uv^T$ satisfies  \eqref{mualternat} and \eqref{main2}.
\end{itemize}
\end{theorem}
We formulate a special case of Theorem \ref{hauptsatz}.
%Flexibler Fall
\begin{corollary}
\label{furinvers}
In addition to the assumptions of Theorem \ref{hauptsatz}, assume that $\dim\ker\Ac(\lambda)=1$ for all $\lambda\in\sigma(\Ac)$. Hence  $m_1(\lambda)=\dim\mathcal{L}_{\lambda}(\Ac)$ holds for all $\lambda\in\sigma(\Ac)$ and $M(\Ac)=n$. Then there exists a rank one pencil $\Pc(s)=(\alpha s-\beta)uv^*$ such that $(\Ac+\Pc)(s)$ is regular and the equations \eqref{mualternat} and \eqref{main2} take the following form
\[
\sigma(\Ac+\Pc)=\{\mu_1,\ldots,\mu_l\}\quad
\text{and}\quad
\dim \mathcal{L}_{\lambda}(\Ac+\Pc)=\begin{cases} m_i, & \text{for $\lambda=\mu_i$,}  \\
0, & \text{for $\lambda\notin\{\mu_1,\ldots,\mu_l\}$.}
\end{cases}
\]
Therefore, for each $\mu_i\in\sigma(\Ac+\Pc)$ there is only one Jordan chain of $(\Ac+\Pc)(s)$ of length $m_i$.
\end{corollary}

Combining Theorem \ref{pert}, Theorem \ref{hauptsatz} and Corollary \ref{furinvers} we get the following result which solves an inverse problem which was investigated for matrices in \cite{K92}.
\begin{theorem}
Given pairwise distinct numbers $\lambda_1,\ldots,\lambda_k\in\overline{\C}$ and $\mu_1,\ldots,\mu_l\in\overline{\C}$ with $k\leq n$, $l\leq n$ and multiplicities $m(\lambda_1),\ldots, m(\lambda_k)$, $m(\mu_1),\ldots,m(\mu_l)
\in\N\setminus\{0\}$
such that
\begin{align*}
    \sum_{i=1}^km(\lambda_i)=\sum_{i=1}^lm(\mu_i)=n.
 \end{align*}
Then, there exists a regular matrix pencil $\Ac(s)\in\C[s]^{n\times n}$ and a rank one pencil $\Pc(s)\in\C[s]^{n\times n}$ such that
\begin{eqnarray*}
\sigma(\Ac)=\{\lambda_1,\ldots,\lambda_k\}, & \quad \quad  & \dim\mathcal{L}_{\lambda_i}(\Ac)=m(\lambda_i),\ i=1,\ldots,k,\\
\sigma(\Ac+\Pc)=\{\mu_1,\ldots,\mu_l\}, &\quad \quad  &   \dim\mathcal{L}_{\mu_i}(\Ac+\Pc)=m(\mu_i),\ i=1,\ldots,l.
\end{eqnarray*}
%\begin{align*}
%    \sum_{i=1}^km(\lambda_i)=\sum_{i=1}^lm(\mu_i)=n.%,\\
% \end{align*}
 %Assume additionally that there are multiplicities $m_{max}(\lambda_i)\leq m(\lambda_i)$ given which denote the length of the longest Jordan chain at $\lambda_i$. Then there exists a regular $n\times n$ matrix pencil $\Ac$ and a rank one pencil $\Pc$ such that \eqref{inv} holds and additionally
%\begin{align}
 %   m_{max}(\lambda_i)=m_1(\lambda_i),\ i=1,\ldots,k%\quad   m_{max}(\lambda_i)=m_1(\lambda_i)
%\end{align}
%if and only if \eqref{summen} holds and additionally
%\begin{align}
%\sum_{\mu_i\notin\sigma(\Ac)}m(\mu_i)\leq \sum_{i=1}^km_{max}(\lambda_i)\\
 %  m(\lambda_j)-m_{max}(\lambda_j)\leq m(\mu_i)\leq m(\lambda_j)+\sum_{i=, i\neq j1}^lm_{max}(\lambda_i),\quad  \text{for $\mu_i=\lambda_j$}
%\end{align}
\end{theorem}
\begin{lemma}
\label{surjective}
Let $\Ac(s)$ be a regular matrix pencil satisfying Assumption \ref{nijann}. Then for the polynomial $m_{\Ac}(s)$ defined in \eqref{minpol}
%\begin{align}
%\label{degminpol}
%\deg m_J=M(\Ac)-m_1(\infty)=\sum_{i=1}^mm_1(\lambda_i)
%\end{align}
and for every $p(s)\in\C[s]$ with
\[
\deg p\leq M(\Ac)-1%\deg m_J+m_1(\infty)-1
\]
there exist $u,v\in\C^n$ such that
\begin{align}
\label{polDS}
p(s)=v^*m_{\Ac}(s)(sE-A)^{-1}u.
\end{align}
If $E$ and $A$ are real matrices and $p(s)\in\R[s]$, then there exists $u,v\in\R^{n}$ satisfying \eqref{polDS}.
%that has a Weierstra\ss\ canonical form with respect to $\K$ with matrices $J\in\K^{r\times r},N\in\K^{(n-r)\times(n-r)}$, then %and the  space of polynomials $\pol_{m}^{\K}(s)$ with coefficients in $\K$ and degree less or equal to $m\in\N$,
\end{lemma}
\begin{proof}%[Proof of Lemma \ref{surjective}]
%First observe that the representation of $m_J$ given in \eqref{minpol} implies
%\[
%\deg m_J=M(\Ac)-m_1(\infty)=\sum_{i=1}^mm_1(\lambda_i)
%\]
We introduce
\[
\Theta_{\Ac}:\C^n\times\C^n\rightarrow \big\{p(s)\in\C[s] ~|~ \deg p\leq M(\Ac)-1\big\},\quad (u,v)\mapsto v^*m_{\Ac}(s)(sE-A)^{-1}u
\]
and show the surjectivity of this map. Since the surjectivity of $\Theta_{\Ac}$ is invariant under equivalence transformations of the form of $sE-A$ to $S(sE-A)T$ with invertible $S,T\in\C^{n\times n}$, we can assume that $\Ac(s)$ is given in Weierstra\ss\ canonical form \eqref{wnf} with matrices $J$ and $N$. If $\sigma(J)=\{\lambda_1,\ldots,\lambda_m\}$ for some complex numbers  $\lambda_1,\ldots,\lambda_m$, then $J$ and $N$ are given by
\begin{align}
\label{JN}
J=\bigoplus_{i=1}^{m}\bigoplus_{j=1}^{\dim\ker\Ac(\lambda_i)}J_{m_j(\lambda_i)}(\lambda_i),\quad N=\bigoplus_{j=1}^{\dim\ker\Ac(\infty)}J_{m_j(\infty)}(0)
\end{align}
with Jordan blocks $J_{k}(\lambda)$ of size $k$ at $\lambda\in\C$ given by
\[
J_{k}(\lambda)=\begin{pmatrix}\lambda & 1 &  & & \\ & \cdot & \cdot  &  &  \\ &&\cdot&\cdot& \\ & & &\cdot &1 \\ &  &   & & \lambda \end{pmatrix}\in\C^{k\times k}.
\]
This allows us to simplify the resolvent representation with $u=(u_0^*,u_1^*)^*$,  $v=(v_0^*,v_1^*)^*$, $u_0,v_0\in\C^r$ and $u_1,v_1\in\C^{n-r}$ to
\begin{align}
& ~~~~v^*m_{\Ac}(s)(sE-A)^{-1}u =v_0^*m_{\Ac}(s)(sI_r-J)^{-1}u_0+v_1^*m_{\Ac}(s)(sN-I_{n-r})^{-1}u_1 \nonumber\\[1ex] &=v_0^*m_{\Ac}(s)\bigoplus_{\tiny\begin{matrix}i=1,\ldots,m,\\ j=1,\ldots,\dim\ker\Ac(\lambda_i)\end{matrix}}\begin{pmatrix}(s-\lambda_i)^{-1} &% -(s-\lambda_i)^{-2} &
\ldots & (-1)^{-m_j(\lambda_i)+1}(s-\lambda_i)^{-m_j(\lambda_i)}
\\% 0 & (s-\lambda_i)^{-1} & \ldots & (-1)^{-n_j(\lambda_i)}(s-\lambda_i)^{-n_j(\lambda_i)+1} \\
  & \ddots & \vdots \\  &  & (s-\lambda_i)^{-1}\end{pmatrix}u_0\nonumber \\[1ex]&~~~+v_1^*m_{\Ac}(s)\bigoplus_{ j=1,\ldots,\dim\ker\Ac(\infty)}\begin{pmatrix}-1 & -s & \ldots & -s^{m_j(\infty)-1} \\  & -1 & \ldots & -s^{m_j(\infty)-2} \\
&  & \ddots & \vdots \\  &  &  & -1\end{pmatrix}u_1.\label{resJ}
\end{align}
Observe $M(\Ac)=\deg m_{\Ac}+m_1(\infty)$. From \eqref{wirklichnijann} and \eqref{minpol} we see that $\Theta_{\Ac}$ maps into the set
\begin{align*}
%\label{image}
\{p(s)\in\C[s] ~|~ \deg p\leq M(\Ac)-1\}.
\end{align*}
Obviously, the right hand side of \eqref{resJ} consists of the sum of products involving two block matrices. Consider the first summand, then the  entries of the1 first row of the blocks of the block matrix for $j=1$ and $i=1,\ldots,m$ are linearly independent as they are functions with a pole in $\lambda_i$ of order from one up to $m_1(\lambda_i)$. Choosing suitable contours and applying the residue theorem, one sees that
\[
\{(s-\lambda_i)^{-r} ~|~ i=1,\ldots,m,~ r=1,\ldots,m_1(\lambda_i)\}.
\]
is also linearly independent. After multiplication with $m_{\Ac}(s)$ this set of functions remains linearly independent. Therefore the set
\begin{align*}
%\label{P1}
P_1:=\{m_{\Ac}(s)(s-\lambda_i)^{-r} ~|~ i=1,\ldots,m,~ r=1,\ldots,m_1(\lambda_i)\}
\end{align*}
is linearly independent and it contains $\sum_{i=1}^mm_1(\lambda_i)=\deg m_{\Ac}$ elements, each of degree less or equal to $\deg m_{\Ac}-1$. Moreover, for $j=1$, the entries of the first row of the blocks in the block matrix of the second summand on the right hand side of \eqref{resJ} are linearly independent and form the linearly independent set of polynomials
\begin{align*}
%    \label{P2}
P_2:=\{m_{\Ac}(s)s^r ~|~ r=0,\ldots,m_1(\infty)-1\}
\end{align*}
which contains $m_1(\infty)$ elements of degree between $\deg m_{\Ac}$ and $\deg m_{\Ac}+m_1(\infty)-1=M(\Ac)-1$. Hence $P_1\cup P_2$ consists of
\begin{align}
\label{doppelherz}
\deg m_{\Ac}+m_1(\infty)=M(\Ac)
\end{align}
linearly independent elements. Furthermore, one can choose certain entries of $v=(v_0^*,v_1^*)$ as one and all others as zero such that the multiplication with $v_0^*$ and $v_1^*$ in \eqref{resJ} picks exactly  the first row of each block with $j=1$ for all $i=1,\ldots,m$. By choosing one entry of $u$ as one and all others as zero we see
%Furthermore by choosing one entry of $u$ and $v$ as one and all others as zero we see
\begin{align*}
%\label{polMenge}
P_1\cup P_2\subset \ran \Theta_{\Ac}
\end{align*}
and from the linearity of the map $u\mapsto\Theta_{\Ac}(u,v)$ with \eqref{doppelherz} the lemma is proved for matrices $E, A\in\C^{n\times n}$.

We consider the case where $E,A$ are real matrices.
Here we use the Weierstra\ss\ canonical form over $\R$, obtained in \cite{G59}, with transformation matrices $S,T\in\R^{n\times n}$. The matrix $N$ is the same as in \eqref{JN} and $J$ is in real Jordan canonical form, see \cite[Section 3.4.1]{HJ13}, %Therefore $N$ has the same representation as in \eqref{JN} but for $J$ we have to replace the Jordan blocks corresponding to non-real eigenvalues, which leads to (cf.\ \cite[Theorem 3.4.1.5]{HJ13}))
\begin{align*}
%\label{JNreell}
J=\bigoplus_{\tiny \begin{matrix}\lambda\in\sigma(J),\\ \im\lambda> 0\end{matrix}}\bigoplus_{j=1}^{\dim\ker\Ac(\lambda)}J_{m_j(\lambda)}^{\R}(\lambda)\oplus\bigoplus_{\lambda\in\sigma(J)\cap\R}\bigoplus_{j=1}^{\dim\ker\Ac(\lambda)}J_{m_j(\lambda)}(\lambda),
\end{align*}
where $J_{m_j(\lambda)}(\lambda)$, $\lambda\in\sigma(\Ac)\cap\R$, are Jordan blocks of size $m_j(\lambda)$ and
$J_{l}^{\R}(\lambda)\in\R^{2l\times 2l}$ for some $l\in\N\setminus\{0\}$ is a real Jordan block at $\lambda=a+ib$ with $a\in\R$, $b>0$,
given by
\[
J_l^{\R}(\lambda):=\begin{pmatrix} C(a,b) & I_2 &  & & \\ & \cdot & \cdot  &  &  \\ &&\cdot&\cdot& \\ & & &\cdot &I_2 \\ &  &   & & C(a,b) \end{pmatrix}\in\R^{2l\times 2l},\quad C(a,b):=\begin{pmatrix}a & b \\ -b& a\end{pmatrix}\in\R^{2\times 2}.
\]
%For $\K=\R$, the definition of the real Jordan blocks in \eqref{JNreell} for $\lambda=a+ib$ %recall that a
%Jordan block corresponding to the eigenvalue pair $a\pm ib$ with $a,b\in\R$ is given by (cf. \cite[Section 3.4.1]{HJ13})
%\[
%J_k(a,b):=\begin{pmatrix}a & b & 0 & 0 & \ldots & 0  \\ -b & a & 1 & 0& \dots & 0 \\ 0 & 0 & a & b&  \dots & 0 \\ 0 & 0 & -b & a&  \dots & 0 \\  \vdots & \vdots &  & \ddots & \ddots & \vdots\\ 0 & 0&\ldots & 0 &-b  & a \end{pmatrix}\in\R^{2k\times 2k}.
%J_k^{\R}(\lambda):=\begin{pmatrix}C(a,b) & I_2 & \ldots & 0  \\ 0 & C(a,b) & \ldots & 0 \\ \vdots & \vdots & & I_2   \\ 0 & 0 & \ldots & C(a,b)  \end{pmatrix}\in\R^{2k\times 2k},\quad C(a,b):=\begin{pmatrix}a & b \\ -b& a\end{pmatrix}\in\R^{2\times 2}.
%\]
%implies that the resolvent of this blocks has the following form
Therefore the resolvent of $J_l^{\R}(a,b)$ is given by
\begin{align}
\label{realres}
(sI_{2l}-J_l^\R(a,b))^{-1}=%\begin{pmatrix}\frac{s-a}{(s-a)^2+b^2} & \frac{b}{(s-a)^2+b^2} & \frac{(s-a)b}{((s-a)^2+b^2)^2}  & \frac{b^2}{((s-a)^2+b^2)^2}  & \ldots & \frac{b^k}{((s-a)^2+b^2)^k} \\\frac{-b}{(s-a)^2+b^2} & \frac{s-a}{(s-a)^2+b^2} & \frac{(s-a)^2}{((s-a)^2+b^2)^2}  & \frac{(s-a)b}{((s-a)^2+b^2)^2}  & \ldots & \frac{(s-a)b^{k-1}}{((s-a)^2+b^2)^k}  \\ 0 & 0 & \frac{s-a}{(s-a)^2+b^2} & \frac{b}{(s-a)^2+b^2} &  \dots & \frac{b^{k-1}}{((s-a)^2+b^2)^{k-1}} \\ 0 & 0 & \frac{-b}{(s-a)^2+b^2} & \frac{s-a}{(s-a)^2+b^2}&  \dots & \frac{(s-a)b^{k-2}}{((s-a)^2+b^2)^{k-1}} \\  \vdots & \vdots &  & \ddots & \ddots & \vdots\\ 0 & 0&\ldots & 0 &\frac{-b}{(s-a)^2+b^2} & \frac{s-a}{(s-a)^2+b^2} \end{pmatrix}.
\begin{pmatrix} (s-C(a,b))^{-1} & \ldots &(-1)^{l-1}(s-C(a,b))^{-l} \\
  & \ddots & \vdots \\   &  &  (s-C(a,b))^{-1}\end{pmatrix}
\end{align}
where the entries are given by
\[
(s-C(a,b))^{-k}=((s-C(a,b))^{-1})^k=\left(\frac{1}{(s-a)^2+b^2}\begin{pmatrix}s-a & b \\ -b & s-a\end{pmatrix}\right)^k,\quad k\in\N\setminus\{0\}.
\]
Using the expression \eqref{realres} instead of the blocks occuring in the block matrix in the first summand of the right hand side of  \eqref{resJ} for the non-real eigenvalues, one can define again a linearly independent set of polynomials $P_1$ by picking all first row entries. This set consists again of polynomials all of distinct degree, because the factor $((s-a)^2+b^2)^{m_1(\lambda)}$ occurs in the minimal polynomial $m_{\Ac}(s)$.
The set $P_2$ remains the same as in the complex valued case. Therefore the same arguments imply the surjectivity of $\Theta_{\Ac}$ in this case.
\end{proof}

\begin{proof}[Proof of Theorem \ref{hauptsatz}]
%{\em Proof of Theorem \ref{hauptsatz}}.\
%To proof the Theorem, we show that there exists a pencil $\Pc$ with $\rk\Pc\leq 1$ such that
%\begin{align}
%\label{auchfaktor}
%\chi_{\Ac+\Pc}(s)=q_{\gamma}(s)\cdot\prod_{\lambda\in\sigma(\Ac)\setminus\{\infty\}}(s-\lambda)^{\%dim\mathcal{L}_{\lambda}(\Ac)-m_1(\lambda)}.
%\end{align}
Choose $\alpha,\beta\in\C$ such that $\alpha\neq 0$ and
 \begin{align*}
 %\label{abbed}
 \frac{\beta}{\alpha}\notin\{\mu_1,\ldots,\mu_l\}\cup\sigma(\Ac)%\quad \text{if $\alpha\neq0$ and}\\
% \infty\notin\{\mu_1,\ldots,\mu_l\}\cup\sigma(\Ac),\quad \text{if $\alpha=0$}\nonumber
\end{align*}
holds. We set
\[
\gamma:=m_{\Ac}(\beta/\alpha)\cdot \prod_{\mu_i\in\{\mu_1,\ldots,\mu_l\}\setminus\{\infty\}}(\beta/\alpha-\mu_i)^{-m_i}.%\begin{cases}m_J(\beta/\alpha)\cdot \prod_{i=1}^l(\beta/\alpha-\mu_i)^{-m_i}, & \text{if $\infty\notin\{\mu_1,\ldots,\mu_l\}$,}\\
%m_J(\beta/\alpha)\cdot \prod_{i=1,i\neq j}^l(\beta/\alpha-\mu_i)^{-m_i}, & \text{if $\mu_j=\infty$.}
%\end{cases}
\]
The condition $\beta/\alpha\notin\{\mu_1,\ldots,\mu_l\}\cup\sigma(\Ac)$ implies $m_{\Ac}(\beta/\alpha)\neq0$, hence $\gamma\neq0$. %Otherwise, for $\alpha=0$, set $\gamma:=1$.
We consider
\[
q_{\gamma}(s):=\gamma\prod_{\mu_i\in\{\mu_1,\ldots,\mu_l\}\setminus\{\infty\}}(s-\mu_i)^{m_i}.%\begin{cases}\gamma \prod_{i=1}^l(s-\mu_i)^{m_i}, & \text{if $\infty\notin\{\mu_1,\ldots,\mu_l\}$,} \\
%\gamma\prod_{i=1, i\neq j}^l(s-\mu_i)^{m_i}, & \text{if $\mu_j=\infty$.}\end{cases}
\]
%\\
%Fist observe, that \eqref{neun} and \eqref{zehn} directly follow from (a).\\
As $\sum_{i=1}^lm_i=M(\Ac)$, the polynomial $q_{\gamma}(s)$ satisfies $\deg q_{\gamma}\leq M(\Ac)$. The degree of $m_{\Ac}(s)$ is $M(\Ac)-m_1(\infty)$ which is smaller or equal 
to $M(\Ac)$. %Therefore there exist  $\alpha,\beta\in\C$ such that
From the choice of $\gamma$ we see that $(q_{\gamma}-m_{\Ac})(\beta/\alpha)=0$ holds and therefore
\begin{align*}
%\label{problem}
\frac{q_{\gamma}(s)-m_{\Ac}(s)}{\alpha s-\beta}
\end{align*}
is a polynomial of degree less or equal to $M(\Ac)-1$.
By Lemma \ref{surjective} there exist $u,v\in\C^n$ with
\begin{align*}
%\label{wichtig}
\frac{q_{\gamma}(s)-m_{\Ac}(s)}{\alpha s-\beta}=v^*m_{\Ac}(s)(sE-A)^{-1}u.
\end{align*}
This equation combined with Sylvester's determinant identity leads to
\begin{align}
\label{dannklar}
\frac{q_{\gamma}(s)}{m_{\Ac}(s)}=1+(\alpha s-\beta)v^*(sE-A)^{-1}u=\det(I_n+(sE-A)^{-1}(\alpha s-\beta)uv^*).
\end{align}
%for $s\notin\sigma(\Ac)$.
Now, set $\Pc(s)=(\alpha s -\beta)uv^*$. Then by Lemma \ref{zusatzlemma} $(\Ac+\Pc)(s)$ is regular and from \eqref{dannklar} we obtain %for the pencil \hage{$(\Ac+\Pc)(s)$}
\begin{align*}
\det(\Ac+\Pc)(s)&=\det(sE-A+(\alpha s-\beta)uv^*)\nonumber\\&=\det(sE-A)\det(I_n+(sE-A)^{-1}(\alpha s-\beta)uv^*)%\label{rechterfaktor} 
\\&=\det(sE-A)\frac{q_{\gamma}(s)}{m_{\Ac}(s)}.\nonumber
\end{align*}
Since $\det(sE-A)$ is by Proposition \ref{extralemma} (d) divisible by $m_{\Ac}(s)$, \eqref{mualternat} follows. %Moreover, the characteristic polynomial of $\Ac+\Pc$ is unequal to zero, hence the pencil $\Ac+\Pc$ is regular.
The equation \eqref{main2} follows from Proposition \ref{extralemma} (d) and Theorem \ref{hauptsatz} (a) is proved.
%Since the parameters $u,v\in\C^n$ can be chosen arbitrarily, we can assume without loss of generality that $\Ac(s)=sE-A$ is already in Weierstra\ss\ canonical form \eqref{wnf} with matrices $N\in\C^{(n-r)\times (n-r)}$ and $J\in\C^{r\times r}$.
%Therefore we decompose $u=(u_0^*,u_1^*)^*\in\C^r\times\C^{n-r}$ and $v=(v_0^*,v_1^*)^*\in\C^r\times\C^{n-r}$.
%With Sylvesters determinant identity we rewrite $\chi_{\Ac+\Pc}$ as
%\begin{align}
%\det(\Ac(s) +(\alpha s-\beta)uv^*)&=
%\det(\Ac(s))\det\Big(I_n+(\alpha s-\beta)(sE-A)^{-1}uv^*\Big)
%\nonumber\\&=\frac{\chi_{\Ac}(s)}{m_J(s)}\Big(m_J(s)+(\alpha s-\beta)v^*m_J(s)(sE-A)^{-1}u\Big).
%\label{faktor}
%\end{align}
%Comparing the factorization in this equation for $\chi_{\Ac+\Pc}$ with the one given \eqref{auchfaktor} it remains to show
%\begin{align}
%\label{rechterfaktor}
%m_J(s)+(\alpha s-\beta)v^*m_J(s)(sE-A)^{-1}u)=q_{\gamma}(s).
%\end{align}
%for suitable $u,v\in\C^n$.
%Obviously we can rewrite equation \eqref{rechterfaktor} and choose $\alpha,\beta\in\C$ such that
%\begin{align}
%\label{problem}
%\frac{q_{\gamma}(s)-m_{J}(s)}{\alpha s -\beta}=v^*m_J(s)(sE-A)^{-1}u
%\end{align}
%is a polynomial
% equation \eqref{problem} holds which completes the proof for $\K=\C$.
%The equation \eqref{main2} can then be derived from Theorem \ref{pert}.
The statements in (b) follow by the same construction as above and by Lemma \ref{surjective} for real matrices.
\end{proof}
%For $E,A\in\R^{n\times n}$ the construction of $\Pc$ can be carried out as in (a). %Observe that $\gamma$ given by \eqref{gamma} is real.%, except for the division by $\alpha s-\beta$ in equation \eqref{problem}. %Because for arbitrary  $\gamma\in\C\setminus\{0\}$ the polynomial $q_{\gamma}-m_J$ has not the \textit{real} zero $\beta/\alpha$. Therefore we choose  $t\in\R\setminus\{\mu_1,\ldots,\mu_l\}$ and define
%\[
%\gamma:=\begin{cases}m_J(t)\cdot\prod_{i=1}^l(t-\mu_i)^{-m_i}, & \text{for $\infty\notin\{\mu_1,\ldots,\mu_l\}$,}\\
%m_J(t)\cdot\prod_{i=1,i\neq j}^l(t-\mu_i)^{-m_i}, & \text{for $\mu_j=\infty$.}
%\end{cases}
%\]
%Observing that $\gamma\in\R$ and $(q_{\gamma}-m_J)(\beta/\alpha)=0$ holds completes the proof.

%\begin{bem}
%The construction in the proof of Theorem \ref{hauptsatz} above showed that in particular for any non-zero pair $(\alpha,\beta)\in\C^2$ satisfying \eqref{abbed} one can construct $u,v\in\C^n$ such that $\Pc(s)=(\alpha s-\beta)uv^*$ satisfies \eqref{mualternat} and \eqref{main2}.
%\end{bem}

\section{Eigenvalue placement under parameter restrictions}
\label{rest}
In this section we study the eigenvalue placement under perturbations of the form
\begin{align}
    \label{pertclass}
\Pc(s)=(s u+v)w^*,\quad w\in\C^n
\end{align}
for fixed $u$ and $v$ in $\C^n$ (cf.\ Proposition \ref{DS}). Special cases of this parameter restricted placement are the pole assignment problem for linear systems under state feedback (cf.\ Section \ref{Sapl}) and also the eigenvalue placement problem for matrices under rank one matrices. Obviously, for the perturbation  \eqref{pertclass} the bounds on the multiplicities from Theorem \ref{pert} still hold. But since $u$ and $v$ are now fixed, we obtain tighter bounds which are given below. In the formulation of these bounds we introduce the number  $m_{u,v}(\lambda)$ which basically replaces $m_1(\lambda)$ in the bounds obtained in Theorem \ref{pert}.
%For $\lambda\in\sigma(\Ac)$ we introduce $m_{uv}(\lambda)$.
Assume that the vector-valued function
\begin{align}
\label{poldefi}
s\mapsto(sE-A)^{-1}(s u+v)
\end{align}
has a pole at $\lambda\in\sigma(\Ac)\setminus\{\infty\}$, i.e.\ one of the entries has a pole at $\lambda$. Then, we denote by $m_{u,v}(\lambda)$ the order of $\lambda$ as a pole of \eqref{poldefi}, which is the maximal order of $\lambda$ as a pole of one of the entries. For $\lambda=\infty$ we define $m_{u,v}(\infty)$ as the order of the pole of \[s\mapsto(-sA+E)^{-1}(s v+u)\] at $s=0$.
In the case where no pole occurs we set $m_{u,v}(\lambda)=0$. Hence $m_{u,v}(\lambda)$ is defined for all $\lambda\in\sigma(\Ac)$.
Note that another way of introducing poles and their order is given by the use of the Smith-McMillan form \cite{CPSW91}, \cite[Ch. VI]{G59}, \cite[Ch. S1]{GLR09}.
% and with this number we can state the perturbation bounds resembling those in Theorem \ref{pert}.
\begin{proposition}
Let $\Ac(s)=sE-A$ be regular and let $\Pc(s)$ be given by \eqref{pertclass} with $u,v\in\C^n$ fixed such that $(\Ac+\Pc)(s)$ is regular. %The function $s\mapsto(sE-A)^{-1}(s u+v)$ has a pole at $\lambda\in\sigma(\Ac)\setminus\{\infty\}$. Denote the order of this pole by $m_{uv}(\lambda)$ and denote by $m_{uv}(\infty)$ the pole order of $s\mapsto(-sA+E)^{-1}(s v+u)$ at zero.
%For $\lambda\in\sigma(\Ac)$ we have
%\[
%\dim\mathcal{L}_{\lambda}(\Ac)-\dim\mathcal{L}_{\lambda}(\Ac+\Pc)\leq m_{uv}(\lambda),
%\]
For \[M(\Ac,u,v):=\sum_{\mu\in\sigma(\Ac)}m_{u,v}(\mu)\]
and $\lambda\in\sigma(\Ac)$ we have
\begin{align}
\label{nurdie}
\dim\mathcal{L}_{\lambda}(\Ac)-m_{u,v}(\lambda)\leq\dim\mathcal{L}_{\lambda}(\Ac+\Pc)\leq\dim\mathcal{L}_{\lambda}(\Ac)+M(\Ac,u,v)-m_{u,v}(\lambda).
\end{align}
For $\lambda\in\overline{\C}\setminus\sigma(\Ac)$ we have
\begin{align}
\label{unddie}
0\leq\dim\mathcal{L}_{\lambda}(\Ac+\Pc)\leq M(\Ac,u,v).
\end{align}
From this we obtain
\begin{align*}
\sum_{\lambda\in\sigma(\Ac)}\dim\mathcal{L}_{\lambda}(\Ac+\Pc)\geq n-\mathcal{M}(\Ac,u,v),\\
\sum_{\lambda\in\sigma(\Ac+\Pc)\setminus\sigma(\Ac)}\dim\mathcal{L}_{\lambda}(\Ac+\Pc)\leq \mathcal{M}(\Ac,u,v).
\end{align*}
\end{proposition}
\begin{proof}
%We show for all $\lambda\in\sigma(\Ac)$
%\[
%\dim\mathcal{L}_{\lambda}(\Ac)-m_{uv}(\lambda)\leq\dim\mathcal{L}_{\lambda}(\Ac+\Pc).
%\]
From Sylvester's formula we conclude
\[
\det(\Ac+\Pc)(s)=\det\Ac(s)(1+w^*(sE-A)^{-1}(su+v)).
\]
If $\lambda\in\sigma(\Ac)\setminus\{\infty\}$ then the characteristic polynomial of $\Ac(s)$ has at $\lambda$ a zero with multiplicity $\dim\mathcal{L}_{\lambda}(\Ac)$. Since the order of $\lambda$ as a pole of $s\mapsto 1+w^*(sE-A)^{-1}(su+v)$ is at most $m_{u,v}(\lambda)$, the multiplicity of $\lambda$ as a zero of $\det(\Ac+\Pc)(s)$, hence the dimension of $\mathcal{L}_{\lambda}(\Ac+\Pc)$, can be bounded by
\[
\dim \mathcal{L}_{\lambda}(\Ac+\Pc)\geq \dim\mathcal{L}_{\lambda}(\Ac)-m_{u,v}(\lambda).
\]
This shows the lower bound in \eqref{nurdie}. The upper bound in \eqref{nurdie} for $\lambda\in\sigma(\Ac)\setminus\{\infty\}$ and \eqref{unddie} for $\lambda\in\C\setminus\sigma(\Ac)$ are obtained in the same way as in the proof of Theorem \ref{pert}.
The estimates \eqref{nurdie} and \eqref{unddie} hold also for the dual pencil $\Ac'(s)=-sA+E$ at $\lambda=0$ and, hence, by Proposition \ref{extralemma} (c) the estimates \eqref{nurdie} and \eqref{unddie} also hold for $\Ac(s)$ at $\lambda=\infty$. Now the remaining assertions  follow in the same way as in Theorem \ref{pert}.
\end{proof}

Within these bounds we investigate the placement of the eigenvalues. 
%In the following Theorem \ref{respert} the case where $u$ and $v$ in \eqref{pertclass} are linearly dependent plays a special role. In this case  for $(\alpha,\beta)\in\C^2\setminus\{0\}$ we can write
%\begin{align*}
%\Pc(s)=(\alpha s-\beta)vw^*.
%\end{align*}
We start with a result which can be seen as a generalization of Lemma \ref{zusatzlemma}.
\begin{lemma}
Let $\Ac(s)=sE-A$ be a regular matrix pencil satisfying Assumption \ref{nijann} and let  $u,v,w\in\C^n$ and $\Pc(s)=(su+v)w^*$, such that $(\Ac+\Pc)(s)$ is regular. Then we have
\[
\{\lambda\in\sigma(\Ac) ~|~ \dim\ker\Ac(\lambda)\geq 2\ \text{or}\ m_{u,v}(\lambda)< m_1(\lambda)\}\subseteq\sigma(\Ac+\Pc). 
\]%\marginpar{Ist $\beta/\al
\end{lemma}
\begin{proof}
From Corollary \ref{starr} we deduce 
\[
\{\lambda\in\sigma(\Ac) ~|~ \dim\ker\Ac(\lambda)\geq 2\}\subseteq\sigma(\Ac+\Pc). 
\]
First, we consider the case $\lambda\in\sigma(\Ac)\setminus\{\infty\}$.
The resolvent representation \eqref{resJ} implies that  $s\mapsto(sE-A)^{-1}$ has an entry with a pole at $\lambda$ of order $m_1(\lambda)$. Hence $m_{u,v}(\lambda)\leq m_1(\lambda)$. The condition $m_{u,v}(\lambda)<m_1(\lambda)$,
Sylvester`s determinant formula, and the representation from Proposition \ref{extralemma} (d) lead to
\begin{align}
\begin{split}
\label{sylv}
\det(\Ac+\Pc)(s)&=\det(sE-A)(1+w^*(sE-A)^{-1}(su+v))\\&=(-1)^{n-r}\det(ST)^{-1}m_{\Ac}(s)q(s)(1+w^*(sE-A)^{-1}(su+v)).
\end{split}
\end{align}
Hence $\det(\Ac+\Pc)(s)$ contains the factor $(s-\lambda)$, therefore $\lambda\in\sigma(\Ac+\Pc)$ and the set inclusion is proven for $\lambda\in\sigma(\Ac)\setminus\{\infty\}$.
In the case $\lambda=\infty\in\sigma(\Ac)$, one has to apply the above arguments to the dual pencil $\Ac'(s)=-sA+E$ and $(\Ac+\Pc)'(s)=-sA+E+(sv+u)w^*$ at $\lambda=0$.
\end{proof}

%For this we introduce some notations. %For simplicity we assume that the unperturbed regular pencil $\Ac(s)=sE-A$ is in Weierstrass canonical form. As in the arguments above, one sees that this is no restriction.
%For $\Ac$ regular such that $S\Ac T$ with $S,T\in\C^{n\times n}$ is in Weierstra\ss\  canonical form \eqref{wnf}, the resolvent representation \eqref{resJ} implies that the $i$th entry of $(sSET-SAT)^{-1}S(su+v)$ is a polynomial of the form
%\begin{align}
%\label{zeile}
%p_{i}(s):=\sum_{j=1}^{n_i}(s-\lambda)^{-j}(su_{i,j}+v_{i,j})
%\end{align}
%for $\lambda\in\sigma(\Ac)\setminus\{\infty\}$, $1\leq i\leq n$ and a suitable labelling of $sSu+Sv=(su_{i,j}+v_{i,j})_{i,j}$ and
%\begin{align}
%    \label{zeileinfty}
%p_{ij}(s):=\sum_{j=0}^{n_i}s^{j}(su_{i,j}+v_{i,j})
%\end{align}
 %with $1\leq i\leq n$.
%Denote by $I(\lambda)$ the set of all indices $i=1,\ldots,n$ such that in \eqref{zeile} has a pole of order $m_{uv}(\lambda)$ in the $i$th row.
%For $i\in I(\lambda)$ there are two cases: We have $n_i=m_{uv}(\lambda)$ if and only if
%\[
%-\frac{v_{i,n_i}}{u_{i,n_i}}\neq\lambda
%\]
%or we have $n_i=m_{uv}(\lambda)+1$ if and only if
%\[
%-\frac{v_{i,n_i}}{u_{i,n_i}}=\lambda.
%\]
%For simplicity we use the convention $-\frac{v_{i,n_i}}{u_{i,n_i}}=\infty$ for $u_{i,n_i}=0$. Observe that $i\in I(\lambda)$ implies that $(u_{i,n_i},v_{i,n_i})\neq 0$ and we have $p_{ij}\neq0$ for all indices $i$ and $j$.
\begin{theorem}
\label{respert}
Let $\Ac(s)=sE-A$ be a regular matrix pencil satisfying Assumption \ref{nijann} and let $u,v\in\C^n$ be fixed.
% $\alpha,\beta\in\C$ and $v\in\C^n$ be fixed. 
Choose $\mu_1,\ldots,\mu_l\in\overline{\C}$ with $l\leq M(\Ac,u,v)$ and multiplicities $m_i\in\N\setminus\{0\}$ satisfying $\sum_{i=1}^lm_i=M(\Ac,u,v)$.
If $u,v$ are linearly dependent we make the following further assumptions.
 \begin{itemize}
 \item[\rm (i)]  If $u\neq 0$ and $v=-\mu u$. Then $\mu\notin\sigma(\Ac)$ implies
 $\mu\notin\{\mu_1,\ldots,\mu_l\}$.
 \item[\rm (ii)] If $u=0$ and $v\neq0$. Then $\infty\notin\sigma(\Ac)$ implies $\infty\notin\{\mu_1,\ldots,\mu_l\}$.
 \end{itemize}
%Then the following holds. For every \hage{$\Pc(s)=(su+v)w^*$} with $w\in\C^n$ we have
%\[
%\{\lambda\in\sigma(\Ac) ~|~ \dim\ker\Ac(\lambda)\geq 2\ \text{or}\ m_{u,v}(\lambda)< m_1(\lambda)\}\subseteq\sigma(\Ac+\Pc). 
%\]%\marginpar{Ist $\beta/\alpha$ nicht auch drin?}
Then there exists $w\in\C^n$ such that for $\Pc(s)=(su+v)w^*$ the pencil $(\Ac+\Pc)(s)$ is regular and satisfies
    \begin{align}
        \label{lu}
    \sigma(\Ac+\Pc)=\{\mu_1,\ldots,\mu_l\}\cup\{\lambda\in\sigma(\Ac) | \dim\ker\Ac(\lambda)\geq 2\ \text{or}\ m_{u,v}(\lambda)< m_1(\lambda)\}
    \end{align} 
and the dimensions of the root subspaces are
\begin{align}
\label{mulres}
\dim \mathcal{L}_{\lambda}(\Ac+\Pc)=\begin{cases}\dim \mathcal{L}_{\lambda}(\Ac)-m_{u,v}(\lambda)+m_i, & \text{for $\lambda=\mu_i\in\sigma(\Ac)$,} \\
\dim \mathcal{L}_{\lambda}(\Ac)-m_{u,v}(\lambda), & \text{for $\lambda\in\sigma(\Ac)\setminus\{\mu_1,\ldots,\mu_l\}$,} \\
m_i, & \text{for $\lambda=\mu_i\notin\sigma(\Ac)$,} \\
0, & \text{for $\lambda\notin\sigma(\Ac)\cup\{\mu_1,\ldots,\mu_l\}$.}
\end{cases}
\end{align}
If $E,A$ are real matrices and $u,v\in\R^n$ then we can find $w\in\R^n$ such that \eqref{lu} and \eqref{mulres} hold. 
\end{theorem}
\begin{proof}
%From Corollary \ref{starr} we already know that the $\lambda\in\sigma(\Ac)$ with $\dim\ker\Ac(\lambda)\geq 2$ are contained in the spectrum of $\sigma(\Ac+\Pc)$. We consider only the case  $\lambda\in\sigma(\Ac)\setminus\{\infty\}$.
%The resolvent representation \eqref{resJ} implies that  $s\mapsto(sE-A)^{-1}$ has a pole at $\lambda\in\sigma(\Ac)$ of order $m_1(\lambda)$. Hence $m_{u,v}(\lambda)\leq m_1(\lambda)$. The condition $m_{u,v}(\lambda)< m_1(\lambda)$,
%Sylvester`s determinant formula, and the representation from Proposition \ref{extralemma} (d) \hage{lead to}
%\begin{align}
%\begin{split}
%\label{sylv}
%\det(\Ac+\Pc)(s)&=\det(sE-A)(1+w^*(sE-A)^{-1}(su+v))\\&=(-1)^{n-r}\det(ST)^{-1}m_J(s)q(s)(1+w^*(%sE-A)^{-1}(su+v)).
%\end{split}
%\end{align}
%\hage{Hence} $\det(\Ac+\Pc)(s)$ contains the factor $(s-\lambda)$, therefore $\lambda\in\sigma(\Ac+\Pc)$ and the first statement is shown.
For the proof, we essentially repeat the arguments from the proof of Lemma \ref{surjective}. First,  let us assume that $u$ and $v$ are linearly dependent with $u\neq0$ and, hence, $v=-\mu u$ and $\Pc(s)=(s-\mu)uw^*$. %\hage{Let us first assume that $u$ and $v$ are linearly dependent such that $\Pc(s)=(\alpha s-\beta)uv^*$ with $\alpha,\beta\in\C$ holds.} 
For $S,T\in\C^{n\times n}$ such that $S\Ac(s) T$ is in Weierstra\ss\ canonical form we write in \eqref{sylv}
\begin{align*}
(1+(s-\mu)w^*(sE-A)^{-1}u)=(1+(s-\mu)w^*T(sSET-SAT)^{-1}Su).
\end{align*}
If $\Ac(s)=sE-A$ has only one eigenvalue $\lambda\neq\infty$ with one Jordan chain at $\lambda$ of length $m_1(\lambda)$ and $m_{u,v}(\lambda)\geq 1$, one can easily write down the result of the multiplication of $(sSET-SAT)^{-1}$ with $Su=((Su)_j)_{j=1}^{m_1(\lambda)}$, cf.\ \eqref{resJ},
%Then the resolvent representation \eqref{resJ} implies for a 
\begin{align}
\label{ausmulti}
(sSET-SAT)^{-1}Su=\begin{pmatrix} \sum_{j=k}^{m_1(\lambda)}(-1)^{j-k}(s-\lambda)^{-j+k-1}(Su)_{j}\end{pmatrix}_{k=1}^{m_1(\lambda)}.%,\quad 1\leq k\leq m_1(\lambda).
\end{align}
Denote by $m\in\N$ with $1\leq m\leq m_1(\lambda)$ the largest index such that  $(Su)_m\neq 0$. Such an  index exists because of $m_{u,v}(\lambda)\geq 1$ and then $(Su)_k=0$ for all $m<k\leq m_1(\lambda)$.
Let us now assume that $\mu\notin\sigma(\Ac)$. 
%We consider only the eigenvalues $\lambda\in\sigma(\Ac)\setminus\{\infty\}$ with $m_{u,v}(\lambda)\geq 1$, because for $m_{u,v}(\lambda)=0$ the expressions in \eqref{ausmulti} are zero. 
The assumption $\mu\notin\sigma(\Ac)$ implies that $m_{u,v}(\lambda)$ is not only the order of $\lambda$ as a pole of $s\mapsto (s-\mu)(sE-A)^{-1}u$ but also the order of $\lambda$ as a pole of $s\mapsto (sE-A)^{-1}u$. Therefore  $m=m_{u,v}(\lambda)$ hence $(Su)_{m_{u,v}(\lambda)}\neq0$ and we consider the following functions given by the right hand side of \eqref{ausmulti}
%Therefore there exists an entry in \eqref{ausmulti} which has a pole in $\lambda$ and no entry in \eqref{ausmulti} has a pole of higher order. 
%We assume further that $m_{u,v}(\lambda)=m_1(\lambda)$ holds. This implies $(Su)_{m_1(\lambda)}\neq0$ and we consider the functions given by the right hand side of \eqref{ausmulti}
\begin{align}
\label{neuebasis}
p_{\lambda,k}(s):=\sum_{j=k}^{m_{u,v}(\lambda)}(-1)^{j-k}(s-\lambda)^{-j+k-1}(Su)_{j}
\end{align}
for $1\leq k \leq m_{u,v}(\lambda)$. The summand in each $p_{\lambda,k}(s)$ with the pole of the highest order has, as $(Su)_{m_{u,v}(\lambda)}\neq0$, a non-zero coefficient so that \eqref{neuebasis} defines a linearly independent set of functions. 
If $\Ac(s)=sE-A$ has only the eigenvalue at $\infty$ with only one Jordan chain of length $m_1(\infty)$ and $m_{u,v}(\infty)\geq 1$ then the largest index $m\in\N$ such that $(Su)_m\neq0$ holds is $m=m_{u,v}(\infty)$ and the right hand side of \eqref{ausmulti} consists of the functions
\[
p_{\infty,k}(s):=-\sum_{j=k}^{m_{u,v}(\infty)-1}s^{j-k}(Su)_{j+1}
\]
for $0\leq k\leq m_{u,v}(\infty)-1$. %The selection of the function $p_{\infty,k}(s)$ in the case $m_{u,v}(\infty)<m_1(\infty)$ can be done in the same way as above. 
%Observe that for each $\lambda\in\sigma(\Ac)$ there are the polynomials
%such that
%\[
%...
%\]
%is linearly independent.

Since $(sSET-SAT)^{-1}$ has block diagonal structure \eqref{resJ}, similar expressions as in \eqref{ausmulti} occur when $\Ac(s)$ has more than one Jordan chain at $\lambda\in\sigma(\Ac)$ or more than one eigenvalue. Here there exists for all $\lambda\in\sigma(\Ac)$ with $m_{u,v}(\lambda)\geq1$ some block on the right hand side of \eqref{resJ} such that the above construction of the functions  $p_{\lambda,k}(s)$ can be carried out. %for all eigenvalues $\lambda\in\sigma(\Ac)$ with $m_{u,v}(\lambda)\geq 1$.
As in the proof of Lemma \ref{surjective} we multiply with a polynomial
\[
\widetilde{m}(s):=\prod_{\lambda\in\sigma(\Ac)\setminus\{\infty\}}(s-\lambda)^{m_{u,v}(\lambda)}
\]
and conclude that the set
\begin{align}
\begin{split}
&~~~\{p_{\lambda,k}(s)\widetilde{m}(s)~|~ \lambda\in\sigma(\Ac)\setminus\{\infty\},\ 1\leq k\leq m_{u,v}(\lambda)\}\\&\cup\{p_{\infty,k}(s)\widetilde{m}(s)~|~ \infty\in\sigma(\Ac),\ 0\leq k\leq m_{u,v}(\infty)-1\} \label{polmengeherz}
\end{split}
\end{align}
is linearly independent and consists of  $\sum_{\lambda\in\sigma(\Ac)}m_{u,v}(\lambda)=M(\Ac,u,v)$ polynomials of degree at most $M(\Ac,u,v)-1$ since we consider only those  $\lambda\in\sigma(\Ac)$ with $m_{u,v}(\lambda)\geq 1$. Therefore it is a basis in the space of polynomials of degree less or equal to $M(\Ac,u,v)-1$.
%Here we consider the polynomials
%The polynomials in this row of $\tilde{w}^*(sSET-SAT)^{-1}(s\tilde{u}+\tilde{v})$ are given by
%Hence it is a basis of a subspace of co-dimension one in the space of polynomials of degree less or equal $M(\Ac,u,v)$.
By (i) we have $\mu\notin\{\mu_1,\ldots,\mu_l\}$ and we consider
\[
\gamma:=\widetilde{m}(\mu)\prod_{\mu_i\in\{\mu_1,\ldots,\mu_l\}\setminus\{\infty\}}(\mu-\mu_i)^{-m_i}
%\begin{cases}\tilde{m}_J(\beta/\alpha)\cdot\prod_{i=1}^l(\beta/\alpha-\mu_i)^{-m_i}, & \text{if $\infty\notin\{\mu_1,\ldots,\mu_l\}$,}\\ \tilde{m}_J(\beta/\alpha)\cdot\prod_{i=1, i\neq j}^l(\beta/\alpha-\mu_i)^{-m_i}, & \text{if $\mu_j=\infty$,}\end{cases}
\]
and the polynomial
\begin{align*}
%\label{qgamma}
\widetilde{q}_{\gamma}(s):=\gamma \prod_{\mu_i\in\{\mu_1,\ldots,\mu_l\}\setminus\{\infty\}}(s-\mu_i)^{m_i}.%\begin{cases}\gamma\prod_{i=1}^l(s-\mu_i)^{m_i}, & \text{if $\infty\notin\{\mu_1,\ldots,\mu_l\}$,}\\
%\gamma\prod_{i=1, i\neq j}^l(s-\mu_i)^{m_i}, & \text{if $\mu_j=\infty$.}
%\end{cases}
\end{align*}
Since $T$ is invertible and from the linear independence of the polynomials there exists $w\in\C^n$ such that
\begin{align}
\label{division}
\frac{\widetilde{q}_{\gamma}(s)-\widetilde{m}(s)}{s-\mu}=w^*T\widetilde{m}(s)(sSET-SAT)^{-1}Su
\end{align}
holds. Plugging this into \eqref{sylv} proves \eqref{lu} and \eqref{mulres} in the case $\mu\notin\sigma(\Ac)$. 
%Assume now that $\mu\in\sigma(\Ac)$ then we have by (i) that $\mu\in\{\mu_1,\ldots,\mu_l\}$ and it is no restriction to assume $\mu=\mu_1$. Then by carrying out the same construction as above and by choosing any $\gamma\in\C\setminus\{0\}$ we see that  $\widetilde{q}_{\gamma}(\mu)=\widetilde{m}(\mu)=0$. Hence the left hand side of \eqref{division} yields a polynomial of degree at most $M(\Ac,u,v)-1$. This implies that there exists  $w\in\C^n$ such that \eqref{division} holds in the case .

Let us now assume that $\mu\in\sigma(\Ac)$ and we consider again the case that $\Ac(s)=sE-A$ has only one Jordan chain at $\lambda$ of length $m_1(\lambda)$ and $m_{u,v}(\lambda)\geq1$. This means that $\mu=\lambda$ and the definition of  $m_{u,v}(\lambda)$ implies $m=m_{u,v}(\lambda)+1$ and we consider the following functions given by the right hand side of \eqref{ausmulti}
\begin{align}
\label{neuebasis2}
p_{\lambda,k}(s):=\sum_{j=k}^{m_{u,v}(\lambda)+1}(-1)^{j-k}(s-\lambda)^{-j+k-1}(Su)_{j}
\end{align}
for $1\leq k\leq m_{u,v}(\lambda)+1$. These functions define a linearly independent set. %and no function has a pole of order higher than $m_{u,v}(\lambda)-1$. 
As in  the previous sub case, we are looking for a solution $w\in\C^{n}$ with $n=m_{u,v}(\lambda)+1$ and  $(w_k)_{k=1}^n:=w^*T$ of the equation
\begin{align}
\label{neuegl}
\prod_{\mu_i\in\{\mu_1,\ldots,\mu_l\}\setminus\{\infty\}}(s-\mu_i)^{m_i}-\widetilde{m}(s)&=w^*T\widetilde{m}(s)(s-\mu)(sSET-SAT)^{-1}Su\\ &=\sum_{k=1}^{m_{u,v}(\lambda)+1}w_k \widetilde{m}(s)(s-\mu)p_{\lambda,k}(s).\nonumber
\end{align}
If we set $w_1=\ldots=w_{m_{u,v}(\lambda)}=0$ and $w_{m_{u,v}(\lambda)+1}=-((Su)_{m_{u,v}(\lambda)+1})^{-1}$ the right hand side of \eqref{neuegl} is equal to $-\widetilde{m}(s)$. On the other hand, by the linear independence of $\{\widetilde{m}(s)p_{\lambda,k}(s)\}_{k=1}^{m_{u,v}(\lambda)+1}$ in the space of polynomials of degree at most $m_{u,v}(\lambda)=M(\Ac,u,v)$, there is a solution $w\in\C^n$ of the equation 
\[
\prod_{\mu_i\in\{\mu_1,\ldots,\mu_l\}\setminus\{\infty\}}(s-\mu_i)^{m_i}=\sum_{k=1}^{m_{u,v}(\lambda)+1}w_k\widetilde{m}(s)p_{\lambda,k}(s)
\]
This, together with the linearity of equation \eqref{neuegl} in $w$ implies the existence of $w\in\C^n$ such that \eqref{neuegl} holds which proves the assertions in this sub case, that there is only one eigenvalue $\lambda\neq\infty$. The block diagonal structure of $(sSET-SAT)^{-1}$ as in \eqref{resJ} implies the existence of $w\in\C^n$ such that \eqref{lu} and \eqref{mulres} hold in the case when $\Ac(s)$ has more than one Jordan chain at  $\lambda\in\sigma(\Ac)$ or more than one eigenvalue. One only has to replace the functions $p_{\lambda,k}(s
)$ for $\lambda=\mu$ with those defined in \eqref{neuebasis2}.

%Now the existence of $w\in\C^n$ such that \eqref{lu} and \eqref{mulres} hold when there is more than one eigenvalue, \hage{$\infty\in\sigma(\Ac)$} and with more than one Jordan chain, follows again from block diagonal structure of $(sSET-SAT)^{-1}$ as in \eqref{resJ}. One only has to replace the functions $p_{\lambda,k}(s
%)$ for $\lambda=\mu$ with those defined in \eqref{neuebasis2}.
%Now, from the linearity of \eqref{neuegl} in $w$ we see that the linear independence of the polynomials 
%can be solved by setting $w_1=0$ and choosing the coefficients $w_2,\ldots,w_{m_{u,v}(\lambda)+1}$ in the right way.

%Next, we consider the case that $\sigma(\Ac)=\{\infty\}$ with $k(\infty)=1$. In the case $u\neq 0$ we have from the invertibility of $S$ that $Su\neq0$.
Next, when $u$ and $v$ are linearly dependent with $u=0$ then we can further assume that $v\neq0$ because $v=0$ implies that $m_{u,v}(\lambda)=0$ for all  $\lambda\in\sigma(\Ac)$ and then \eqref{lu} and \eqref{mulres} hold trivially. This case can be treated by considering the dual pencils $\Ac'(s)$ and $(\Ac+\Pc)'(s)$, because for the dual pencils we are in the case of (i) with $\mu=0$. Therefore, we have already proven that \eqref{lu} and \eqref{mulres} hold for the dual pencils under the assumption that $0\notin\sigma(\Ac')$ implies $0\notin\{\mu_1,\ldots,\mu_l\}$. By Proposition \ref{extralemma} (c) this condition follows from the assumption (ii) that $\infty\notin\sigma(\Ac)$ implies $\infty\notin\{\mu_1,\ldots,\mu_l\}$. Furthermore we see immediately that  $S\Ac'(s)T$ is block diagonal and that $\dim\mathcal{L}_{\lambda}(\Ac)$ for $\lambda\in\sigma(\Ac)\setminus\{0\}$ is equal to $\dim\mathcal{L}_{\lambda^{-1}}(\Ac')$. This proves \eqref{lu} and \eqref{mulres} for $\Ac(s)$ and $(\Ac+\Pc)(s)$.
This finishes the proof of the theorem given that $u$ and $v$ are linearly dependent.

%Bei den anderen Eigenwerten hat lässt man die polynome ja gleich...
%From the block structure we obtain a set of polynomials at each eigenvalue $\lambda\in\sigma(\Ac)$ with $m_{u,v}(\lambda)\geq 1$. Now in the equation 
%\[
%.
%\]
%Since $p_{\lambda,1}$ is by $(Su)_{m_{u,v}(\lambda)}\neq0$ a non-zero constant polynomial. On
%one can choose the coefficient in front of p_{\lambda}

%In the case 
%Let us now assume that $\infty\in\sigma(\Ac)$ and $u=0$ then ...
%For the remaining statements with $u=0$ and the assumption (ii) consider the dual pencils $\Ac'(s)$ and $(\Ac+\Pc)'(s)$.

In the case where $u$ and $v$ are linearly independent we have in \eqref{sylv} 
\begin{align*}
(1+w^*(sE-A)^{-1}(su+v))=(1+w^*T(sSET-SAT)^{-1}(sSu+Sv)).
\end{align*}
If $\Ac(s)=sE-A$ has only one eigenvalue $\lambda\neq\infty$ with one Jordan chain of length  $m_1(\lambda)$ and $m_{u,v}(\lambda)\geq 1$ then the product $(sSET-SAT)^{-1}(sSu+Sv)$ is given by %for $\lambda\in\sigma(\Ac)\setminus\{\infty\}$ such that $m_{u,v}(\lambda)\geq 1$ we obtain the functions
\begin{align}
\label{ausmulti2}
\begin{pmatrix} \sum_{j=k}^{m_1(\lambda)}(-1)^{j-k}(s-\lambda)^{-j+k-1}(s(Su)_{j}+(Sv)_j)\end{pmatrix}_{k=1}^{m_1(\lambda)}.
\end{align}
This again allows us to define a linearly independent set of polynomials. For this we consider the largest index $m\in\N$ with $1\leq m\leq m_1(\lambda)$ such that $s(Su)_m+(Sv)_m\neq 0$. This implies that $s(Su)_k+(Sv)_k=0$ for all $m<k\leq m_1(\lambda)$ and we obtain the following to cases.
%From the defintion of $m_{u,v}(\lambda)\geq 1$ we have that $(s(Su)_{m_{u,v}(\lambda)}+(Sv)_{m_{u,v}(\lambda)})\neq0$.
If  $(s(Su)_{m}+(Sv)_{m})$ and $(s-\lambda)$ are linearly independent in $\C[s]$ then we have  $m=m_{u,v}(\lambda)$ and we consider the following entries of \eqref{ausmulti2} 
\[
p_{\lambda,k}(s):=\sum_{j=k}^{m_{u,v}(\lambda)}(-1)^{j-k}(s-\lambda)^{-j+k-1}(s(Su)_{j}+(Sv)_j)%\sum_{j=1}^k(-1)^{j-1}(s-\lambda)^{-j}(su_{\lambda, k-j}+v_{\lambda, k-j})
\]
for $1\leq k\leq m_{u,v}(\lambda)$. If $(s(Su)_{m}+(Sv)_{m})$ and $(s-\lambda)$ are linearly dependent then we have $m=m_{u,v}(\lambda)+1$ and define 
\[
p_{\lambda,k}(s):=\sum_{j=k}^{m_{u,v}(\lambda)+1}(-1)^{j-k}(s-\lambda)^{-j+k-2}(s(Su)_{j}+(Sv)_j)%\sum
\]
 for $1\leq k\leq m_{u,v}(\lambda)+1$. In both cases we have from the condition $(s(Su)_{m}+(Sv)_{m})\neq 0$ that the functions $p_{\lambda,k}(s)\neq0$ define a linearly independent set. 
If $\Ac(s)$ has only the eigenvalue $\infty$ with one Jordan chain of length $m_1(\infty)$ we consider in the case $(Su)_{m_{u,v}(\infty)}\neq0$
\[
p_{\infty,k}(s):=-\sum_{j=k}^{m_{u,v}(\infty)-1}s^{j-k-1}(s(Su)_{j+1}+(Sv)_{j+1})%-\sum_{j=0}^ks^{j}(su_{\lambda, k-j}+v_{\lambda, k-j})\neq 0
\]
for $0\leq k\leq m_{u,v}(\infty)-1$ and if $(Su)_{m_{u,v}(\infty)}=0$
\[
p_{\infty,k}(s):=-\sum_{j=k}^{m_{u,v}(\infty)-1}s^{j-k}(s(Su)_{j+1}+(Sv)_{j+1})%-\sum_{j=0}^ks^{j}(s
\]
for $0\leq k\leq m_{u,v}(\infty)-1$. If $\Ac(s)$ has more than one eigenvalue and with more than one Jordan chain, then we can use again the block diagonal structure of $(sSET-SAT)^{-1}$ to define the functions $p_{\lambda,k}(s)$ for all $\lambda\in\sigma(\Ac)$. The linear independence of these functions follows from the linear independence of $u$ and $v$.
 
Now the same linear independency arguments from above can be applied. This implies the existence of $w\in\C^n$ satisfying the equation
\[
\prod_{\mu_i\in\{\mu_1,\ldots,\mu_l\}\setminus\{\infty\}}(s-\mu_i)^{m_i}-\widetilde{m}(s)=w^*T\widetilde{m}(s)(sSET-SAT)^{-1}S(su+v).
\]
For real matrices $E$ and $A$ one uses the transformation to the real Weierstra\ss\ canonical form with $S,T\in\R^{n\times n}$ such that for $\lambda=a+ib\in\sigma(\Ac)\setminus\R$ with $b>0$ we use the blocks given by \eqref{realres} to define the polynomials $p_{\lambda,k}(s)$. These polynomials have to be inserted into \eqref{polmengeherz} which gives a linearly independent set. Since the transformation matrices $S,T$ are real, this proves the statement in the real case.
\end{proof}

If $E=I_n$ then
%Associating with $A\in\C^{n\times n}$ the pencil $\Ac(s)=sI_n-A$,
$\sigma(\Ac)$ equals the set $\sigma(A)$ of all eigenvalues of $A$ and $\infty\notin\sigma(\Ac)$ holds. For the perturbation class given for a fixed $v\in\C^n$ by
\[
\Pc(s):=vw^*,\quad w\in\C^n,
\]
Theorem \ref{respert} for $u=0$ leads to the following eigenvalue placement result for matrices which is a generalization of the placement results obtained in \cite{DY07,G73,K92,T76}. %\marginpar{Hier noch eine Referenz warum es well-known ist...}
\begin{corollary}
Let $\Ac(s)=sI_n-A$ be a matrix pencil satisfying Assumption \ref{nijann} with minimal polynomial $m_A(s)$ of the matrix $A\in\C^{n\times n}$ and fixed $v\in\C^n$. Then for any given values $\mu_1,\ldots,\mu_l\in\C$ with $l\leq M(\Ac,0,v)$ and multiplicities $m_i\in\N\setminus\{0\}$ satisfying $\sum_{i=1}^lm_i=M(\Ac,0,v)$ there exists $w\in\C^n$ such that 
\[
\sigma(A+vw^*)=\{\mu_1,\ldots,\mu_l\}\cup\{\lambda\in\sigma(A) ~|~ \dim \ker (\lambda I_n-A)\geq 2,\ m_{0,v}(\lambda)<m_1(\lambda)\}
\]
and \eqref{mulres} holds. If $v\in\C^n$ is not fixed, then one can always choose $v$ such that $m_{0,v}(\lambda)=m_1(\lambda)$ holds for all $\lambda\in\sigma(A)$ implying that $M(\Ac,0,v)=\deg m_A$. For this choice of $v$ there exists $w\in\C^n$ such that 
\[
\sigma(A+vw^*)=\{\mu_1,\ldots,\mu_l\}\cup\{\lambda\in\sigma(A) ~|~ \dim \ker (\lambda I_n-A)\geq 2\}.
\]
If $E,A$ are real matrices then $u,v$ can be chosen in $\R^n$.
\end{corollary}

\section{Application to single input differential-algebraic equations with feedback}
\label{Sapl}
Parameter restricted perturbations of the form \eqref{pertclass} occur naturally in the study of differential algebraic equations with a single input given by $E,A\in\C^{n\times n}$, $b\in\C^n$, $x_0\in\C^n$ and the equation
\begin{align}
\label{Dae}
\frac{d}{dt}Ex(t)=Ax(t)+bu(t),\quad t\in[0,\infty),\quad x(0)=x_0.
\end{align}
For this equation we consider a state feedback of the form $u(t)=f^*x(t)$ with $f\in\C^n$.
It is well known that the solution of the closed loop-system
\begin{align*}
%\label{clDae}
\frac{d}{dt}Ex(t)=(A+bf^*)x(t),\quad t\in[0,\infty),\quad x(0)=x_0
\end{align*}
can be expressed with the eigenvalues and Jordan chains of the matrix pencil $sE-(A+bf^*)$, see \cite{BIT12}. This pencil can be written as a perturbation of $sE-A$ with the rank one pencil $\Pc(s)=-bf^*$.
By fixing $u=0$ and $v=-b$ in \eqref{pertclass} we can write
\[
\Pc(s)=-bf^*=(su+v)f^*.
\]
For $E$ singular we have $\infty\in\sigma(\Ac)$ and for $E$ invertible we have $\infty\notin\sigma(\Ac)$.

In \cite{BR13} it was shown that a system given by $(E,A,b)$ with $sE-A$ regular is controllable in the behavioural sense if and only if the Hautus condition
\[
\rk[\lambda E-A,b]=n,\quad \text{for all $\lambda\in\C$}
\]
holds. A transformation to Weierstra\ss\ canonical form \eqref{wnf} reveals that this Hautus condition is equivalent to $\dim\ker\Ac(\lambda)=1$ for all $\lambda\in\C\cap\sigma(\Ac)$ and
\[
m_{0,-b}(\lambda)=m_1(\lambda)=\dim\mathcal{L}_{\lambda}(\Ac),\quad \text{for all $\lambda\in\C\cap\sigma(\Ac)$.}
\]
Hence Theorem \ref{respert} implies that all finite eigenvalues can by placed arbitrarily in $\C$.

The feedback placement problem for regular matrix pencils was studied in \cite{C81,L86,LO85,M93}.
%stabilizable in the behavioural sense if and only if the Hautus condition
%\[
%\rk[\lambda E-A,b]=n,\quad \text{for all $\lambda\in\C,  \im\lambda\geq0$.}
%\]
%holds. KÖNNTE MAN JETZT ANALOG FORMULIEREN
In the following, as in \cite{LO85}, we do not assume controllability.
Then Theorem \ref{respert} implies the following.
%But for single input systems their conditions can be simplified.
 %Therefore we obtain an eigenvalue assignment result as in \cite{C81,L86,LO85} under simpler conditions because we are in the single input case. %As in \cite{OL85} we need no assumptions on the controllability of the system \eqref{}
%As in the ODE-case it is well known (Achim und Mehrmann zitieren) that the solutions can be expressed in terms of the exponential functions. Here the exponential functions contain the finite eigenvalues of the associated pencil $\Ac(s)=sE-A$.
%Therefore we want to describe how the eigenvalues can be placed under feedback. Since the rank of $E$ remains constant for every feedback, the number of eigenvalues at $\infty$ ($\dim\mathcal{L}_{\infty}$) also remains constant.
%Describe possible spectrum under feedback}
%From the results in the previous sections we obtain the following result. Note that we do not assume controllability. In the matrix case it is known, that the controllability implies arbitrary placeability of the eigenvalues, hence no multiple Jordan chains at a certain eigenvalue are allowed, i.e. $k(\lambda)\leq1$ for $\lambda\in\C$.
%With the ideas from the previous sections we can describe the possible change of eigenvalues and Jordan chains under such perturbations.
\begin{theorem}
Let $(E,A,b)$ be a system given by \eqref{Dae} such that $\Ac(s)=sE-A$ is a regular matrix pencil  satisfying Assumption \ref{nijann}. %The function $s\mapsto(sE-A)^{-1}b$ has a pole at $\lambda\in\sigma(\Ac)\setminus\{\infty\}$. Denote the order of this pole by $m(\lambda)$.
%\begin{itemize}
 %   \item[\rm (a)] For $\lambda\in\sigma(\Ac)$ we have
  %  \begin{align}
   % \label{eins}
    %\dim\mathcal{L}_{\lambda}(sE-A)-\dim\mathcal{L}_{\lambda}(sE-(A+bf^*))\leq m(\lambda),
    %\end{align}
    %and
    %\begin{align}
    %\label{infty}
    %\dim\mathcal{L}_{\infty}(sE-(A+bf^*))=n-r.
    %\end{align}
    %\item[\rm (b)] Introduce $\mathcal{M}(\Ac,b):=\sum_{\lambda\in\sigma(\Ac)\setminus\{\infty\}}m(\lambda)$ then we have for $\lambda\in\sigma(\Ac)\setminus\{\infty\}$
    %\begin{align*}
    %   \dim\mathcal{L}_{\lambda}(\Ac)-m(\lambda)\leq\dim\mathcal{L}_{\lambda}(\Ac-bf^*)\leq \dim\mathcal{L}_{\lambda}(\Ac)+(\mathcal{M}(\Ac,b)-m(\lambda))
    %\end{align*}
    %and for $\lambda\in\C\setminus\sigma(\Ac)$
    %\begin{align*}
     %   0\leq\dim\mathcal{L}_{\lambda}(\Ac-bf^*)\leq \mathcal{M}(\Ac,b).
    %\end{align*}
    %Moreover we have
    %%\[
    %\sum_{\lambda\in\sigma(\Ac)}\dim\mathcal{L}_{\lambda}(\Ac-bf^*)\geq n-\mathcal{M}(\Ac,b),\quad \sum_{\lambda\in\sigma(\Ac-bf^*)\setminus\sigma(\Ac)}\dim\mathcal{L}_{\lambda}(\Ac-bf^*)\leq M(\Ac,b).
    %\]
%\[
%\mathcal{M}(\lambda)=\max\{k : (s-\lambda)^{-k} \text{appears the $i$-th row in \eqref{resJ} and $(T^*b)_i\neq0$ } \}\leq m_1(\lambda)
%\]
%This bound is sharp.
%\item[\rm (c)]
Choose pairwise distinct numbers $\mu_1,\ldots,\mu_l\in\overline{\C}$ with $l\leq M(\Ac,0,-b)$. Choose multiplicities $m_1,\ldots,m_l\in\N\setminus\{0\}$ with  $\sum_{i=1}^lm_i=M(\Ac,0,-b)$. Then:
\begin{itemize}
\item[\rm (a)] For $E$ singular there exists a feedback vector $f\in\C^n$, such that $sE-(A+bf^*)$ is regular and $\sigma(sE-(A+bf^*))$ equals to
\begin{align}
\label{ende}
\{\mu_1,\ldots,\mu_l\}\cup\{\lambda\in\sigma(\Ac) ~|~ \dim\ker\Ac(\lambda)\geq 2\ \text{or}\ m_{0,-b}(\lambda)< m_1(\lambda)\}.
\end{align}
\item[\rm (b)] For $E$ invertible and $\infty\notin\{\mu_1,\ldots,\mu_l\}$ there exists a feedback vector $f\in\C^n$ such that $sE-(A+bf^*)$ is regular and $\sigma(sE-(A+bf^*))$ equals \eqref{ende}.
\end{itemize}
In both cases {\rm (a)} and {\rm (b)} the
 dimensions of the root subspaces are given by formula \eqref{mulres}. If $E,A$ are real matrices and $b\in\R^n$ then $f$ can be chosen in $\R^n$.
\end{theorem}
%\begin{proof}
%From Sylvesters determinant identity we have
 %   \begin{align}
  %      \label{sylvgähn}
   % \det(sE-(A+bf^*))=\det(sE-A)(1-f^*(sE-A)^{-1}b).
    %\end{align}
    %The definition of $m(\lambda)$ implies that the function $s\mapsto f^*(sE-A)^{-1}b$ has a pole of order less or equal to $m(\lambda)$ for all $f\in\C^n$. Since $s=\lambda$ is a zero of order $\dim\mathcal{L}_{\lambda}(\Ac)$, then by \eqref{sylvgähn} the zero order can be bounded from below by
    %\[
    %\dim\mathcal{L}_{\lambda}(\Ac-bf^*)\geq\dim\mathcal{L}_{\lambda}(\Ac)-m(\lambda).
    %\]
    %This proves \eqref{eins}.
    %The equation \eqref{infty} is obvious since the leading coefficient of the pencil remains constant under the feedback. Now (b) can be proven like Theorem \ref{pert} using \eqref{eins}.
%\end{proof}

\section*{Acknowledgments}
The authors wish to thank the anonymous referees for their careful reading and for valuable comments which improved the quality of the manuscript. The authors also thank A.C.M.\  Ran for pointing out useful references.

 \nocite{*}

\bibliographystyle{amsplain}

\end{document}